\documentclass[12pt]{article} 
\usepackage{amsfonts,amsmath,latexsym,amssymb,mathrsfs,amsthm,comment}
\usepackage{slashbox}
\usepackage{caption}
\let\OLDthebibliography\thebibliography
\renewcommand\thebibliography[1]{
  \OLDthebibliography{#1}
  \setlength{\parskip}{3pt}
  \setlength{\itemsep}{0pt plus 0.3ex}
}

\def\numberlikeadb{\global\def\theequation{\thesection.\arabic{equation}}}
\numberlikeadb
\newtheorem{theorem}{Theorem}[section]
\newtheorem{lemma}[theorem]{Lemma}
\newtheorem{corollary}[theorem]{Corollary}

\newtheorem{proposition}[theorem]{Proposition}
\newtheorem{remark}[theorem]{Remark}
\newtheorem{example}[theorem]{Example}

\usepackage{lscape}
\usepackage{caption}
\usepackage{multirow}

\usepackage{color} 

\begin{document} 

%Wasserstein and Kolmogorov error bounds for variance-gamma approximation via Stein's method II

\title{Stein factors for variance-gamma approximation in the Wasserstein and Kolmogorov distances}

\author{Robert E. Gaunt\footnote{Department of Mathematics, The University of Manchester, Oxford Road, Manchester M13 9PL, UK. Email: robert.gaunt@manchester.ac.uk}}

\date{} 
\maketitle

\vspace{-5mm}

\begin{abstract}
We obtain new bounds for the solution of the variance-gamma (VG) Stein equation that are of the correct form for approximations in terms of the Wasserstein and Kolmorogorov metrics.  These bounds hold for all parameters values of the four parameter VG class.  As an application we obtain explicit Wasserstein and Kolmogorov distance error bounds in a six moment theorem for VG approximation of double Wiener-It\^{o} integrals.
\end{abstract}


\noindent{{{\bf{AMS 2010 Subject Classification:}}} Primary 60F05; 62E17

\section{Introduction} 

The variance-gamma (VG) distribution with parameters $r > 0$, $\theta \in \mathbb{R}$, $\sigma >0$, $\mu \in \mathbb{R}$ has probability density function
\begin{equation}\label{vgdefn}p(x) = \frac{1}{\sigma\sqrt{\pi} \Gamma(\frac{r}{2})} \mathrm{e}^{\frac{\theta}{\sigma^2} (x-\mu)} \bigg(\frac{|x-\mu|}{2\sqrt{\theta^2 +  \sigma^2}}\bigg)^{\frac{r-1}{2}} K_{\frac{r-1}{2}}\bigg(\frac{\sqrt{\theta^2 + \sigma^2}}{\sigma^2} |x-\mu| \bigg), 
\end{equation}
with support $\mathbb{R}$.  In the limit $\sigma\rightarrow 0$ the support becomes the region $(\mu,\infty)$ if $\theta>0$, and is $(-\infty,\mu)$ if $\theta<0$. Here $K_\nu(x)$ is a modified Bessel function of the second kind, defined in Appendix \ref{appendix}. For a random variable $Z$ with density (\ref{vgdefn}), we write $Z\sim\mathrm{VG}(r,\theta,\sigma,\mu)$.    
Different parametrisations are given in \cite{eberlein} and the book \cite{kkp01}, in which they refer to the distribution as the generalized Laplace distribution. 
%Distributional properties are given in \cite{gaunt vg} and Chapter 4 of the book \cite{kkp01}.  

%For purposes of exposition, we set $\mu=0$ throughout this paper; the results for the general case follow easily from a simple translation.

The VG distribution is widely used in financial modelling \cite{madan2,madan}; an overview of this and other applications are given in \cite{kkp01}.
% (in which they use the name generalized Laplace distribution). 
The VG distribution also has a rich distributional theory (see Chapter 4 of \cite{kkp01} and \cite{gaunt vg}), and the class contains several classical distributions as special or limiting cases, such as the normal, gamma, Laplace, product of zero mean normals and difference of gammas (see Proposition 1.2 of \cite{gaunt vg} for a list of further cases).  

%Because of these desirable distributional properties, the VG distributions arise as the limit distribution in many probabilistic limit theorems. 

Stein's method \cite{stein} is a powerful and widely used approach for deriving quantitative limit theorems in probability.  Originally developed for normal approximation, it has been extended to several other distributions such as the Poisson \cite{chen75}, exponential \cite{chatterjee,pekoz1}, gamma \cite{gaunt chi square, luk} and Laplace \cite{pike}; for an overview see \cite{ley}. Stein's method was extended to the VG distribution by \cite{gaunt thesis, gaunt vg}, with subsequent technical advances made by \cite{gaunt vgii}. The Malliavin-Stein method \cite{np09,np12} for VG approximation was developed by \cite{eichelsbacher}, and together with results from \cite{gaunt thesis, gaunt vg}, they were able to obtain ``six moment" theorems for the VG approximation of double Wiener-It\^{o} integrals. Recently, with the aid of results from this paper (Theorem \ref{thm1} and Corollary \ref{cortothm1}), \cite{aet21} have achieved a six moment for the VG approximation of double Wiener-It\^{o} integrals with optimal convergence rate. Their result complements and in some senses generalises the celebrated optimal fourth moment theorem of \cite{np15} for normal approximation; see also \cite{aek20} for a similar result for gamma approximation. Further VG approximations are given in \cite{aaps17,ag18,azmoodeh,bt17}, in which the limiting distributions can be represented as the difference of two centered gamma random variables or the product of two mean zero normal random variables.

%  The class of VG distributions contain many classical distributions as special or limiting cases, such as the normal, gamma, Laplace, product of zero mean normals and difference of gammas (see Proposition 1.2 of \cite{gaunt vg} for a list of further cases).  Consequently, the VG distribution appears in many other settings beyond financial mathematics \cite{kkp01}; for example, in alignment-free sequence comparison \cite{lippert, waterman}.  In particular, starting with the works \cite{gaunt thesis, gaunt vg}, Stein's method \cite{stein} has been developed for VG approximation.  The theory of \cite{gaunt thesis, gaunt vg} and the Malliavin-Stein method (see \cite{np12}) was applied by \cite{eichelsbacher} to obtain ``six moment" theorems for the VG approximation of double Wiener-It\^{o} integrals.  Further VG approximations are given in \cite{aaps17} and \cite{bt17}, in which the limiting distribution is the difference of two centered gamma random variables.

The starting point of Stein's method for VG approximation is the Stein equation \cite{gaunt vg}
\begin{equation}
\label{377}L_{r,\theta,\sigma,\mu}f(x):=\sigma^2(x-\mu)f''(x)+(\sigma^2r+2\theta (x-\mu))f'(x)+(r\theta-(x-\mu))f(x)=\tilde{h}(x),
\end{equation}
where $\tilde{h}(x)=h(x)-\mathbb{E}h(Z)$ for $h:\mathbb{R}\rightarrow\mathbb{R}$ and $Z\sim\mathrm{VG}(r,\theta,\sigma,\mu)$, and is such that $\mathbb{E}|h(Z)|<\infty$.  Here $L_{r,\theta,\sigma,\mu}$ is the VG Stein operator. Along with the Stein equations of \cite{pekoz} and \cite{pike}, this was one of the first second order Stein equations in the literature.  Let us now set $\mu=0$; we recover the general case using the translation relation that if $Z\sim\mathrm{VG}(r,\theta,\sigma,\mu)$ then $Z-\mu\sim\mathrm{VG}(r,\theta,\sigma,0)$.  The solution to (\ref{377}) is (see \cite[Lemma 3.3]{gaunt vg}; here we have used a different parametrisation for the VG distribution)
\begin{align}  f_h(x) &=-\frac{\mathrm{e}^{-\theta x/\sigma^2}}{\sigma^2|x|^{\nu}}K_{\nu}\bigg(\frac{\sqrt{\theta^2+\sigma^2}}{\sigma^2}|x|\bigg) \int_0^x \mathrm{e}^{\theta t/\sigma^2} |t|^{\nu} I_{\nu}\bigg(\frac{\sqrt{\theta^2+\sigma^2}}{\sigma^2}|t|\bigg) \tilde{h}(t) \,\mathrm{d}t \nonumber \\
\label{vgsolngeneral0} &\quad-\frac{\mathrm{e}^{-\theta x/\sigma^2}}{\sigma^2|x|^{\nu}}I_{\nu}\bigg(\frac{\sqrt{\theta^2+\sigma^2}}{\sigma^2}|x|\bigg) \int_x^{\infty} \mathrm{e}^{\theta t/\sigma^2} |t|^{\nu} K_{\nu}\bigg(\frac{\sqrt{\theta^2+\sigma^2}}{\sigma^2}|t|\bigg)\tilde{h}(t)\,\mathrm{d}t,
\end{align}
where $\nu=\frac{r-1}{2}$ and $I_\nu(x)$ is a modified Bessel function of the first kind, defined in Appendix \ref{appendix}. If $h$ is bounded, then $f_h(x)$ and $f_h'(x)$ are bounded for all $x\in\mathbb{R}$, and (\ref{vgsolngeneral0}) is the unique bounded solution when $r\geq1$, and the unique solution with bounded first derivative for $r>0$ (see \cite[Lemma 3.14]{gaunt thesis}).

One may approximate a random variable of interest $W$ by a VG random variable $Z\sim\mathrm{VG}(r,\theta,\sigma,0)$ by evaluating both sides of (\ref{377}) at $W$, taking expectations and then taking the supremum of both sides over a class of functions $\mathcal{H}$ to arrive at
\begin{equation*}\sup_{h\in\mathcal{H}}|\mathbb{E}h(W)-\mathbb{E}h(Z)|=\sup_{h\in\mathcal{H}}\mathbb{E}|L_{r,\theta,\sigma,0}f_h(W)|.
\end{equation*}
This is important because many standard probability metrics have a representation of the form $\sup_{h\in\mathcal{H}}|\mathbb{E}h(W)-\mathbb{E}h(Z)|$.  In particular, taking  
\begin{align*}\mathcal{H}_{\mathrm{K}}&=\{\mathbf{1}(\cdot\leq z)\,|\,z\in\mathbb{R}\}, \\
\mathcal{H}_{\mathrm{W}}&=\{h:\mathbb{R}\rightarrow\mathbb{R}\,|\,\text{$h$ is Lipschitz, $\|h'\|\leq1$}\}, \\
\mathcal{H}_{\mathrm{BW}}&=\{h:\mathbb{R}\rightarrow\mathbb{R}\,|\,\text{$h$ is Lipschitz, $\|h\|\leq1$ and $\|h'\|\leq1$}\}
\end{align*}
yields the Kolmogorov, Wasserstein and bounded Wasserstein distances. We shall denote the Kolmogorov and Wasserstein distances by $d_{\mathrm{K}}$ and $d_{\mathrm{W}}$, respectively. Here, and throughout the paper, $\|\cdot\|:=\|\cdot\|_\infty$ denotes the usual supremum norm of a real-valued function on $\mathbb{R}$.

In order for the above procedure to be effective, it is crucial to have suitable bounds on the solution (\ref{vgsolngeneral0}), which are often referred to in the literature as Stein factors. This is technically demanding, due to the presence of modified Bessel functions in the solution  together with the singularity at the origin in the Stein equation (\ref{377}). The first bounds in the literature \cite{gaunt vg} (given for the case $\theta=0$) resulted from a brute force approach that involved the writing of three papers on modified Bessel functions \cite{gaunt ineq1, gaunt ineq2,gaunt diff} and long calculations (given in Section 3.3 and Appendix D of the thesis \cite{gaunt thesis}). A significant advance was later made by \cite{dgv15}. Their iterative approach  reduced the problem of bounding derivatives of arbitrary order to bounding just the solution and its first derivative. Consequently, they were able to obtain bounds for derivatives of any order for the whole class of VG distributions. However, the dependence of the bounds of \cite{dgv15} on the test function $h$ meant that they were only suitable for approximation in metrics that are weaker than the Wasserstein and Kolmogorov metrics. A technical advance was made in the recent work \cite{gaunt vgii} in which the iterative technique of \cite{dgv15} and new inequalities for integrals of modified Bessel functions \cite{gaunt ineq3} were used to obtain bounds suitable for Wasserstein 
%($\|h'\|$) 
and Kolmogorov
%($\|\tilde{h}\|$) 
distance error bounds in the case $\theta=0$. 

In this paper, we complement the work of \cite{gaunt vgii} by obtaining analogous bounds for the whole class of VG distributions. Our results have been made possible due to a very recent work on inequalities for integrals of modified Bessel functions \cite{gaunt ineq2020}.  
%Given the added complexity of the $\theta\not=0$ case, our bounds are not as simple as those of \cite{gaunt vgii} and do not have an optimal dependence on the parameters (see Remark \ref{remop}), but crucially they have the same correct dependence on the test function $h$.
%and are less accurate (bigger numerical constants and worse dependence on the parameters $r$ and $\sigma$) than the bounds of \cite{gaunt vgii}, but crucially they have the same dependence on the test function $h$. 
The bounds we establish have followed from a series of contributions to the problem of bounding derivatives of solutions of Stein equations together with technical results for modified Bessel functions spanning several papers, and the overall task of establishing such bounds for VG approximation has arguably been more demanding than for any other distribution for which this step of Stein's method has been achieved. 
Like a number of other papers in the literature, for example \cite{bgx15,bu98,bx06,cd19,daly08,daly17,dgv15,lefevre,lmx20,mg16,
roellin12}, the main focus of this paper is to obtain new Stein factors, although we do present a simple application to complement the work of \cite{eichelsbacher} on the Malliavin-Stein method for VG approximation.  Here our work fixes a technical issue and provides explicit constants for their quantitative limit theorems, with Corollary \ref{cor5.4} giving a quantitative sixth moment theorem for the VG approximation of double Wiener-It\^{o} integrals, with explicit bounds in the Wasserstein and Kolmogorov distances. We give a demonstration of these general bounds by using them to obtain bounds on the rate of convergence, with respect to the Wasserstein and Kolmogorov distances, in one of the main results of the recent work \cite[Theorem 2.4]{bt17}, which concerns the generalized Rosenblatt process at extreme critical exponent. Here the limiting distribution is a VG distribution with $\theta\not=0$, highlighting the importance of our generalisation of the Stein factor bounds of \cite{gaunt vgii} to the general $\theta\in\mathbb{R}$ case.
%. Our work complements the recent paper of \cite{aaps17}, in which bounds on the rate of convergence were obtained in the Wasserstein-2 distance. 
%As with other contributions to the literature on Stein factors, the greatest impact of our results may arise from their use in future unforeseen applications. 
%For example, the author has been made aware \cite{ehsan} of such a forthcoming work in which bounds from Theorem \ref{thm1} and Corollary \ref{cortothm1} are used in the proof of the main result.

%An example of this is the forthcoming work of \cite{aet20} in which the Stein factor bounds of Theorem \ref{thm1} and Corollary \ref{cortothm1} are used in their proof of a beautiful sixth moment theorem for the VG approximation of double Wiener-It\^{o} integrals with optimal convergence rate (see also \cite{np15} and \cite{aek20} for similar results for Gaussian and gamma approximation, respectively). Their upper bound has a faster convergence rate than the bound (\ref{doww2}) of Corollary \ref{cor5.4} with respect to a smooth probability metric that is weaker than the Wasserstein metric.

%Their  work complements the recent results of \cite{np15} and \cite{aek20} in which similar results have been obtained for Gaussian and gamma approximation, respectively.

The rest of this paper is organised as follows.  In Section \ref{sec2}, we present some basic properties of VG distributions that will be used in the paper.  In Section \ref{sec3}, we obtain our new bounds for the solution of the VG Stein equation.  We also provide a connection between Kolmogorov and Wasserstein distances between a general distribution and a VG distribution. Our application to the Malliavin-Stein method for VG approximation is given in Section \ref{sec4}.  Proofs of some technical results are given in Section \ref{sec6}. Appendix \ref{appendix} lists some relevant basic properties and inequalities for modified Bessel functions.  Appendix \ref{appb} provides a list of uniform bounds for expressions involving integrals of modified Bessel functions that we use in obtaining our bounds for the solution of the VG Stein equation.

\section{The class of variance-gamma distributions}\label{sec2}

In this section, we present some basic properties of the class of variance-gamma (VG) distributions that will be useful in the remainder of the paper; for further properties, see \cite{gaunt vg} and Chapter 4 of \cite{kkp01}.

%The parameter $r$ is known as the \emph{scale} parameter. As $r$ increases, the distribution becomes more rounded around its peak value $\mu$ (see below). The parameter $\sigma$ is called the \emph{tail} parameter. As $\sigma$ decreases, the tails decay more quickly (see (\ref{pjv})).  The parameter $\mu$ is the \emph{location} parameter. Calculations can often be simplified by using the basic relation that if $Z\sim \mathrm{SVG}(r,1,0)$, then $\sigma Z+\mu \sim \mathrm{SVG}(r,\sigma,\mu)$.  
%The $\mathrm{SVG}(r,1,0)$ distribution is in a sense the \emph{standard} symmetric variance-gamma distribution.

The modified Bessel function in the probability density function (\ref{vgdefn}) makes it difficult to parse on first inspection.  We can gain some understanding from the following limiting forms.  Applying the limiting form (\ref{Ktendinfinity}) to (\ref{vgdefn}) gives that,
\begin{equation*}
%\label{pjv}
p(x)\sim \frac{1}{2^{\frac{r}{2}}(\theta^2+\sigma^2)^{\frac{r}{4}}\Gamma(\frac{r}{2})}|x|^{\frac{r}{2}-1}\exp\bigg(\frac{\theta}{\sigma^2}(x-\mu)-\frac{\sqrt{\theta^2+\sigma^2}}{\sigma^2}|x-\mu|\bigg), \quad |x|\rightarrow\infty,
\end{equation*}
which is valid for all $r>0$, $\theta\in\mathbb{R}$, $\sigma>0$ and $\mu\in\mathbb{R}$. Similarly, this time using the limiting form (\ref{Ktend0}), we have that (see \cite{gaunt thesis}) 
\begin{equation}\label{pmutend}p(x)\sim\begin{cases}\displaystyle \frac{1}{2\sigma\sqrt{\pi}(1+\theta^2/\sigma^2)^{\frac{r-1}{2}}}\frac{\Gamma\big(\frac{r-1}{2}\big)}{\Gamma\big(\frac{r}{2}\big)}, &  x\rightarrow\mu,\:r>1, \\
\displaystyle -\frac{1}{\pi\sigma}\log|x-\mu|, & x\rightarrow\mu,\: r=1, \\
\displaystyle \frac{1}{(2\sigma)^r\sqrt{\pi}}\frac{\Gamma\big(\frac{1-r}{2}\big)}{\Gamma\big(\frac{r}{2}\big)}|x-\mu|^{r-1}, & x\rightarrow\mu,\: 0<r<1. \end{cases}
\end{equation}
We see that the density has a singularity at $x=\mu$ if $r\leq1$.  In fact, for all parameter values, the $\mathrm{VG}(r,\theta,\sigma,\mu)$ distribution is unimodal.  The following properties of the mode $M$ can be found in \cite{gm20}. For $0<r\leq2$, $\theta\in\mathbb{R}$, $\sigma>0$, or $r>0$, $\theta=0$, $\sigma>0$ the mode is equal to $\mu$.  Suppose now that $r>2$, $\theta\in\mathbb{R}$, $\sigma>0$. Then $M=\mu+\mathrm{sgn}(\theta)\cdot x^*$, where $x^*$ is the unique positive solution of the equation 
\begin{equation}\label{xstar}K_{\frac{r-3}{2}}\bigg(\frac{\sqrt{\theta^2+\sigma^2}}{\sigma^2}x\bigg)=\frac{|\theta|}{\sqrt{\theta^2+\sigma^2}}K_{\frac{r-1}{2}}\bigg(\frac{\sqrt{\theta^2+\sigma^2}}{\sigma^2}x\bigg).
\end{equation}
For $r>2$, $\theta>0$, $\sigma>0$, we have the two-sided inequality 
\begin{equation}\label{gmineq}\theta(r-3)_+<M-\mu<\theta(r-2),
\end{equation}  
with the inequality reversed for $\theta<0$. Here $x_+=\max\{0,x\}$.

The following result is new and needed in the proof of Proposition \ref{prop1}. Both Proposition \ref{prop2mode} and Proposition \ref{prop1} are proved in Section \ref{sec6}.

\begin{proposition}\label{prop2mode} $r>2$, $\theta\in\mathbb{R}$, $\sigma>0$, $\mu\in\mathbb{R}$ the $\mathrm{VG}(r,\theta,\sigma,\mu)$ density (\ref{vgdefn}) can be bounded above for all $x\in\mathbb{R}$ by
\begin{align}\label{vgpdfineq2}p(x) \leq \frac{\Gamma(\frac{r-1}{2})}{2\sigma\sqrt{\pi} \Gamma(\frac{r}{2})}\bigg(\frac{\sigma^2}{\theta^2+\sigma^2}\bigg)^{\frac{r-1}{2}} \mathrm{e}^{\frac{\theta^2}{\sigma^2} (r-2)}.
\end{align}
For  $r>3$, $\theta\not=0$, $\sigma>0$, $\mu\in\mathbb{R}$, the following bound improves on (\ref{vgpdfineq2}) for all $x\in\mathbb{R}$:
\begin{align}\label{vgpdfineq}p(x)< \frac{1}{\sigma\sqrt{\pi} \Gamma(\frac{r}{2})} \mathrm{e}^{\frac{\theta^2}{\sigma^2} (r-2)} \bigg(\frac{\theta(r-3)}{2\sqrt{\theta^2 +  \sigma^2}}\bigg)^{\frac{r-1}{2}} K_{\frac{r-1}{2}}\bigg(\frac{\theta\sqrt{\theta^2 + \sigma^2}}{\sigma^2} (r-3) \bigg).
\end{align}
\end{proposition}

\begin{remark}The proof of inequality (\ref{vgpdfineq}) in Proposition \ref{prop2mode} makes use of the two-sided inequality (\ref{gmineq}).  An alternative lower bound for the mode of the $\mathrm{VG}(r,\theta,\sigma,\mu)$ distribution is given in \cite[Corollary 2.6]{gm20}, which improves on the lower bound of (\ref{gmineq}) for $r>4$. Applying this inequality in the proof would lead to a more accurate bound than (\ref{vgpdfineq}) for $r>4$, but the resulting bound would be more complicated.  
\end{remark}

%For ease of notation we shall set $\mu=0$; the extension follows on using that $\mu+\mathrm{VG}(r,\theta,\sigma,0)=_d\mathrm{VG}(r,\theta,\sigma,\mu)$.  We shall also restrict our attention to the $\theta>0$ case; the $\theta<0$ follows on using that $\mathrm{VG}(r,-\theta,\sigma,\mu)=_d-\mathrm{VG}(r,\theta,\sigma,\mu)$. 

%To this end, we note that if $(X,Y)$ follows a bivariate gamma distribution with correlation $\rho$ and marginals $X\sim \Gamma(\frac{r}{2},\lambda_1)$ and $Y\sim \Gamma(\frac{r}{2},\lambda_2)$, then the random variable $X-Y$ has the $\mathrm{VG}(r,(2\lambda_1)^{-1}-(2\lambda_2)^{-1},(\lambda_1\lambda_2)^{-1/2}(1-\rho)^{1/2},0)$.  Since here there are four parameters $r$, $\lambda_1$, $\lambda_2$ and $\rho$, one can generate the entire family of $\mathrm{VG}(r,\theta,\sigma,0)$ distributions by varying the parameters $r$, $\lambda_1$, $\lambda_2$ and $\rho$ appropriately.  It is clear that the density of the random variable $X-Y$ will 

%The SVG distribution has a fundamental representation in terms of independent normal and gamma random variables (\cite{kkp01}, Proposition 4.1.2).  Let $X\sim\Gamma(\frac{r}{2},\frac{1}{2})$ (with p.d.f$.$ $\frac{1}{2^{\frac{r}{2}}\Gamma(r/2)}x^{\frac{r}{2}-1}\mathrm{e}^{-x/2}$, $x>0$) and $Y\sim N(0,1)$ be independent.  Then $\mu+\sigma \sqrt{X}Y\sim\mathrm{SVG}(r,\sigma,\mu)$.

The mean and variance of $Z\sim \mathrm{VG}(r,\theta,\sigma,\mu)$ are given by (see \cite{kkp01})
\begin{equation*}\mathbb{E}Z=\mu+r\theta, \quad \mathrm{Var}(Z)=r(\sigma^2+2\theta^2).
\end{equation*}
An application of the Cauchy-Schwarz inequality then yields
\begin{equation}\label{vgabmom}\mathbb{E}|Z-\mu|\leq\sqrt{r(\sigma^2+2\theta^2)+r^2\theta^2}.
\end{equation}
Let us write $\mathrm{VG}_c(r,\theta,\sigma)$ for $\mathrm{VG}(r,\theta,\sigma,-r\theta)$, the VG distribution with zero mean. (The notation $\mathrm{VG}_c(r,\theta,\sigma)$ was introduced by \cite{eichelsbacher}, and whilst when introducing the notation they mention that this denotes the $\mathrm{VG}(r,\theta,\sigma,0)$ distribution, it is quite clear when studying their paper that they instead meant $\mathrm{VG}(r,\theta,\sigma,-r\theta)$.) The following formulas for the cumulants of $Y\sim  \mathrm{VG}_c(r,\theta,\sigma)$ \cite[Lemma 3.6]{eichelsbacher} will be used in the proof of Corollary \ref{cor5.4}:
\begin{align*}
&\kappa_2(Y)=r(\sigma^2+2\theta^2), \quad \kappa_3(Y)=2r\theta(3\sigma^2+4\theta^2),\quad \kappa_4(Y)=6r(\sigma^4+8\sigma^2\theta^2+8\theta^4), \\
&\kappa_5(Y)=24r\theta(5\sigma^4+20\sigma^2\theta^2+16\theta^4), \quad \kappa_6(Y)=120r(\sigma^2+2\theta^2)(\sigma^4+16\sigma^2\theta^2+16\theta^4).
\end{align*} 

\begin{comment}
Letting
\begin{equation} \label{parameter} \nu=\frac{r-1}{2}, \quad \alpha =\frac{\sqrt{\theta^2 +  \sigma^2}}{\sigma^2}, \quad \beta =\frac{\theta}{\sigma^2}.
\end{equation}
leads to another useful parametrisation of the VG distribution, which is given in \cite{eberlein}, that will simplify the calculations of our main result Theorem \ref{thm1}. Here $\nu > -1/2$, $\mu \in \mathbb{R}$, $\alpha >|\beta|\geq0$, and the probability density function is given by
\begin{equation} \label{seven} p_{\mathrm{VG}_2}(x) = \frac{(\alpha^2 - \beta^2)^{\nu + \frac{1}{2}}}{\sqrt{\pi} \Gamma(\nu + \frac{1}{2})}\left(\frac{|x-\mu|}{2\alpha}\right)^{\nu} \mathrm{e}^{\beta (x-\mu)}K_{\nu}(\alpha|x-\mu|), \quad x\in \mathbb{R}.
\end{equation}
If a random variable $X$ has density (\ref{seven}), we write $X\sim \mathrm{VG}_2(\nu,\alpha,\beta,\mu)$.
\end{comment}

\section{Bounds for the solution of the Stein equation}\label{sec3}

In this section, we establish new bounds for the solution (\ref{vgsolngeneral0}) of the $\mathrm{VG}(r,\theta,\sigma,\mu)$ Stein equation (\ref{377}) that have the correct dependence on the test function $h$ for the purposes of using Stein's method to derive Wasserstein and Kolmogorov distance error bounds for VG approximation.  Our bounds are valid for the entire parameter space $r>0$, $\theta\in\mathbb{R}$, $\sigma>0$ and $\mu\in\mathbb{R}$.  

We begin by stating two bounds from \cite{dgv15} (see inequalities (3.31) and (3.32) from that reference) that are the only bounds in the current literature that are of a suitable form for deriving Kolmogorov distance bounds for the whole class of VG distributions; no bounds in the literature are suitable for the purposes of obtaining Wasserstein distance bounds for the entire VG class.  We will use these bounds in our proof of Theorem \ref{thm1}. For bounded and measurable $h:\mathbb{R}\rightarrow\mathbb{R}$,
\begin{eqnarray}\label{vgsolnunibound}\|f\|&\leq&\frac{\|\tilde{h}\|}{\sqrt{\theta^2+\sigma^2}}\bigg(\frac{2}{r}+A_{r,\theta,\sigma}\bigg), \\
\label{vgsolnunibound1}\|f'\|&\leq&\frac{\|\tilde{h}\|}{\sigma^2}\bigg(\frac{2}{r}+A_{r,\theta,\sigma}\bigg),
\end{eqnarray}
where
\begin{equation*}\label{breqn}A_{r,\theta,\sigma}=\begin{cases} \displaystyle \frac{2\sqrt{\pi}}{\sqrt{2r-1}}\bigg(1+\frac{\theta^2}{\sigma^2}\bigg)^{\frac{r}{2}}, & \:  r\geq2, \\
\displaystyle 12\Gamma\Big(\frac{r}{2}\Big)\bigg(1+\frac{\theta^2}{\sigma^2}\bigg), & \:   0<r<2. \end{cases}
\end{equation*}
Our presentation of inequalities (\ref{vgsolnunibound}) and (\ref{vgsolnunibound1}) differs a little from that given in \cite{dgv15}. This is discussed in Remark \ref{appbr}.

\begin{comment}
\begin{remark}Our presentation of inequalities (\ref{vgsolnunibound}) and (\ref{vgsolnunibound1}) differs from that given in \cite{dgv15}.  The bounds are a slight simplification of the bounds given on p$.$ 24 of \cite{dgv15} in a different parametrisation (and translated into our $\mathrm{VG}(r,\theta,\sigma,\mu)$ parametrisation on p$.$ 17 of \cite{dgv15} at the cost of two typos).  One of our simplifications is to note that their $\gamma$ (defined on their p$.$ 24) satisfies $|\gamma|<1$.  The other is note that, for $0<r<2$, 
\end{remark}

\begin{equation*}C(r,\theta,\sigma)=\begin{cases} \displaystyle \frac{\sqrt{\pi}\Gamma\left(\frac{r}{2}\right)}{\Gamma\left(\frac{r+1}{2}\right)}\bigg(1+\frac{\theta^2}{\sigma^2}\bigg)^{\frac{r}{2}}, & \:  r\geq2, \\
\displaystyle 6\Gamma\left(\frac{r}{2}\right)\bigg(1-\frac{1}{(1+\sigma^2/\theta^2)^{\frac{1}{2}}}\bigg)^{-\frac{r}{2}}, & \:   0<r<2. \end{cases}
\end{equation*}
\end{comment}

Let us now state our main result.  The theorem extends Theorem 3.1 of \cite{gaunt vgii}, which was given for the $\theta=0$ case, to cover the entire class of VG distributions.

\begin{theorem}\label{thm1}Let $f$ denote the solution (\ref{vgsolngeneral0}) of the $\mathrm{VG}(r,\theta,\sigma,\mu)$ Stein equation (\ref{377}). 

\vspace{2mm}

\noindent{1.} Suppose  $h:\mathbb{R}\rightarrow\mathbb{R}$ is bounded and measurable.  Then
\begin{align}\label{thmone}\|(x-\mu)f(x)\|&\leq \bigg(1+\frac{6}{r}+B_{r,\theta,\sigma}\bigg)\|\tilde{h}\|, \\
\label{thmtwo}\|(x-\mu)f'(x)\|&\leq \frac{2\sqrt{\theta^2+\sigma^2}}{\sigma^2}\bigg(1+\frac{6}{r}+B_{r,\theta,\sigma}\bigg)\|\tilde{h}\|, \\
\label{thmthree}\|(x-\mu)f''(x)\|&\leq \frac{1}{\sigma^2}\bigg\{5+2rA_{r,\theta,\sigma}+\bigg(5+\frac{4\theta^2}{\sigma^2}\bigg)\bigg(1+\frac{6}{r}+B_{r,\theta,\sigma}\bigg)\bigg\}\|\tilde{h}\|,
\end{align}
where
\begin{equation*}\label{bdefn}B_{r,\theta,\sigma}=\begin{cases} \displaystyle \sqrt{\frac{\pi(r-1)}{2}}\bigg(1+\frac{\theta^2}{\sigma^2}\bigg)^{\frac{r}{2}-1}, & \:  r\geq2, \\
\displaystyle 2, & \:   0<r<2. \end{cases}
\end{equation*}

\vspace{2mm}

\noindent{2.} Suppose now that $h:\mathbb{R}\rightarrow\mathbb{R}$ is Lipschitz.  Then
\begin{align}\label{thm1f0}\|f\|&\leq\bigg\{4+\frac{2\sqrt{2}}{\sqrt{r}}+\sqrt{2\pi(r+1)}\frac{|\theta|}{\sigma}\bigg(1+\frac{\theta^2}{\sigma^2}\bigg)^{\frac{r-1}{2}}+(\sqrt{2r}+r)A_{r,\theta,\sigma}\bigg\}\|h'\|, \\
\|f'\|\label{thm1f1}&\leq\frac{\sqrt{\theta^2+\sigma^2}}{\sigma^2}C_{r,\theta,\sigma}\|h'\|, \\
\|f''\|\label{thm1f2}&\leq\frac{1}{\sigma^2}\bigg(\frac{2}{r+1}+A_{r+1,\theta,\sigma}\bigg)\bigg\{1+\bigg(2+\frac{\theta^2}{\sigma^2}\bigg)C_{r,\theta,\sigma}\bigg\}\|h'\|,
\end{align}
where
\begin{align}\label{cdefn}C_{r,\theta,\sigma}=6+\frac{2\sqrt{2}}{\sqrt{r}}+2\sqrt{2\pi(r+1)}\frac{|\theta|}{\sigma}\bigg(1+\frac{\theta^2}{\sigma^2}\bigg)^{\frac{r-1}{2}}+2(\sqrt{2r}+r)A_{r,\theta,\sigma}.
\end{align}
We also have that 
\begin{align}
\label{thmfour}\|(x-\mu)f'(x)\|&\leq\bigg(1+\frac{6}{r+1}+B_{r+1,\theta,\sigma}\bigg)\bigg\{1+\bigg(2+\frac{\theta^2}{\sigma^2}\bigg)C_{r,\theta,\sigma}\bigg\}\|h'\|, \\
\label{thmfive}\|(x-\mu)f''(x)\|&\leq\frac{2\sqrt{\theta^2+\sigma^2}}{\sigma^2}\bigg(1+\frac{6}{r+1}+B_{r+1,\theta,\sigma}\bigg)\bigg\{1+\bigg(2+\frac{\theta^2}{\sigma^2}\bigg)C_{r,\theta,\sigma}\bigg\}\|h'\|, \\
\label{thmsix}\|(x-\mu)f^{(3)}(x)\|&\leq\frac{1}{\sigma^2}\bigg\{5+2(r+1)A_{r+1,\theta,\sigma}+\bigg(5+\frac{4\theta^2}{\sigma^2}\bigg)\bigg(1+\frac{6}{r+1}+B_{r+1,\theta,\sigma}\bigg)\bigg\}\nonumber\\
&\quad\times \bigg\{1+\bigg(2+\frac{\theta^2}{\sigma^2}\bigg)C_{r,\theta,\sigma}\bigg\}\|h'\|.
\end{align}
\end{theorem}

\begin{remark}In the light of the singularity in the Stein equation (\ref{377}) at $x=\mu$, the factor of $(x-\mu)$ in the left-hand sides of the estimates (\ref{thmone})--(\ref{thmthree}) and (\ref{thmfour})--(\ref{thmsix}) is quite natural.
\end{remark}

The bounds in Theorem \ref{thm1} can be used together with the iterative technique of \cite{dgv15} to obtain bounds on higher order derivatives of the solution of the VG Stein equation. These bounds will necessarily involve higher order derivatives of the test function $h$; see Proposition \ref{ptpt}.  An example is given in the following corollary, which improves on a bound for $\|f^{(3)}\|$ given on page 18 of \cite{dgv15} by only depending on $\|h''\|$ and $\|h'\|$ (the bound of \cite{dgv15} also had a term involving $\|\tilde{h}\|$).  

%The bound in the corollary, together with some bounds from Theorem \ref{thm1}, will be used in the proof of the main result of \cite{aet20}. 

\begin{corollary}\label{cortothm1}Let $f$ denote the solution (\ref{vgsolngeneral0}) of the $\mathrm{VG}(r,\theta,\sigma,\mu)$ Stein equation (\ref{377}). Let $h:\mathbb{R}\rightarrow\mathbb{R}$ be such that its first derivative $h'$ is bounded and Lipschitz. Then
\begin{align*}\|f^{(3)}\|&\leq \frac{1}{\sigma^4}\bigg(\frac{2}{r+2}+A_{r+2,\theta,\sigma}\bigg)\bigg\{1+\bigg(2+\frac{\theta^2}{\sigma^2}\bigg)C_{r+1,\theta,\sigma}\bigg\}\bigg\{\|h''\|+\\
&\quad+\bigg[\sqrt{\theta^2+\sigma^2}C_{r,\theta,\sigma}+|\theta|\bigg(\frac{2}{r+1}+A_{r+1,\theta,\sigma}\bigg)\bigg\{1+\bigg(2+\frac{\theta^2}{\sigma^2}\bigg)C_{r,\theta,\sigma}\bigg\}\bigg]\|h'\|\bigg\},
\end{align*} 
%\begin{align*}\|f^{(3)}\|&\leq \frac{1}{\sigma^2}\bigg(\frac{2}{r+2}+A_{r+2,\theta,\sigma}\bigg)\bigg\{1+\bigg(2+\frac{\theta^2}{\sigma^2}\bigg)C_{r+1,\theta,\sigma}\bigg\}\bigg\{\|h''\|+\\
%&\quad+\bigg[\frac{\sqrt{\theta^2+\sigma^2}}{\sigma^2}C_{r,\theta,\sigma}+\frac{1}{\sigma^2}\bigg(\frac{2}{r+1}+A_{r+1,\theta,\sigma}\bigg)\bigg\{1+\bigg(2+\frac{\theta^2}{\sigma^2}\bigg)C_{r,\theta,\sigma}\bigg\}\bigg]\|h'\|\bigg\},
%\end{align*} 
where $A_{r,\theta,\sigma}$ and $C_{r,\theta,\sigma}$ are defined as in Theorem \ref{thm1}.
\end{corollary}

\noindent{\emph{Proof of Theorem \ref{thm1}.}} In order to simplify the calculations, we make the following change of parameters
\begin{equation} \label{parameter} \nu=\frac{r-1}{2}, \quad \alpha =\frac{\sqrt{\theta^2 +  \sigma^2}}{\sigma^2}, \quad \beta =\frac{\theta}{\sigma^2}.
\end{equation}
We also set $\mu=0$; the bounds for the general case follow from a simple translation. With these parameters, the solution (\ref{vgsolngeneral0}) can be written as
\begin{align} \label{ink} f(x) &=-\frac{\mathrm{e}^{-\beta x} K_{\nu}(\alpha|x|)}{\sigma^2|x|^{\nu}} \int_0^x \mathrm{e}^{\beta t} |t|^{\nu} I_{\nu}(\alpha|t|) \tilde{h}(t) \,\mathrm{d}t \nonumber \\
&\quad-\frac{\mathrm{e}^{-\beta x} I_{\nu}(\alpha|x|)}{\sigma^2|x|^{\nu}} \int_x^{\infty} \mathrm{e}^{\beta t} |t|^{\nu} K_{\nu}(\alpha|t|)\tilde{h}(t)\,\mathrm{d}t 
\end{align}
\begin{align}
&=-\frac{\mathrm{e}^{-\beta x} K_{\nu}(\alpha|x|)}{\sigma^2|x|^{\nu}} \int_0^x \mathrm{e}^{\beta t} |t|^{\nu} I_{\nu}(\alpha|t|) \tilde{h}(t) \,\mathrm{d}t \nonumber \\
\label{pen}&\quad+\frac{\mathrm{e}^{-\beta x} I_{\nu}(\alpha|x|)}{\sigma^2|x|^{\nu}} \int_{-\infty}^{x} \mathrm{e}^{\beta t} |t|^{\nu} K_{\nu}(\alpha|t|)\tilde{h}(t)\,\mathrm{d}t.
\end{align}
We have equality between the different representations of the solution (\ref{ink}) and (\ref{pen}) because, letting $Z\sim \mathrm{VG}(r,\theta,\sigma,0)$, we have that
\begin{align*}
&\int_{-\infty}^{x} \mathrm{e}^{\beta t} |t|^{\nu} K_{\nu}(\alpha|t|)\tilde{h}(t)\,\mathrm{d}t-\bigg(-\int_x^{\infty} \mathrm{e}^{\beta t} |t|^{\nu} K_{\nu}(\alpha|t|)\tilde{h}(t)\,\mathrm{d}t\bigg)\\
&=\int_{-\infty}^{\infty} \mathrm{e}^{\beta t} |t|^{\nu} K_{\nu}(\alpha|t|)\tilde{h}(t)\,\mathrm{d}t=\frac{1}{\sigma\sqrt{\pi} \Gamma(\frac{r}{2})}  \bigg(\frac{1}{2\sqrt{\theta^2 +  \sigma^2}}\bigg)^{\frac{r-1}{2}}\mathbb{E}[\tilde{h}(Z)]=0, 
\end{align*}
where in the penultimate step we recalled the change of parameters (\ref{parameter}) to observe that $ \mathrm{e}^{\beta t} |t|^{\nu} K_{\nu}(\alpha|t|)$ is the $\mathrm{VG}(r,\theta,\sigma,0)$ density up to the normalising constant.  This equality is useful because 
it means that to obtain uniform bounds for all $x\in\mathbb{R}$ it is sufficient to bound the solution and its derivatives in the region $x\geq0$, provided we consider both the cases of negative and positive $\beta$.  We shall therefore proceed by deriving bounds for $x\geq 0$, which must also hold for $x\leq0$, and thus for all $x\in\mathbb{R}$.

%As noted above, for ease of notation, we set $\sigma=1$ and $\mu=0$.  The bounds for the general case, as stated in the theorem, follow from a simple change of variables; see the proof of Theorem 3.6 of \cite{gaunt vg}.  We also recall that it suffices to obtain bounds in the region $x\geq0$.

Suppose first that $h:\mathbb{R}\rightarrow\mathbb{R}$ is Lipschitz. We begin by proving the bound for $\|f\|$, which will be used in the derivation of some of the other bounds.  The mean value theorem gives that $|\tilde{h}(x)|\leq\|h'\|(|x|+\mathbb{E}|Z|)$, where $Z\sim \mathrm{VG}(r,\theta,\sigma,0)$.  We recall from (\ref{vgabmom}) that $\mathbb{E}|Z|\leq\sqrt{r(\sigma^2+2\theta^2)+r^2\theta^2}$.  Using these two inequalities, together with the integral inequalities (\ref{propb2a12}), (\ref{fff1}), (\ref{rnmt2b}) and (\ref{rnmt1}), gives that, for $x\geq0$,
\begin{align}|f(x)|&\leq\frac{\|h'\|}{\sigma^2}\bigg\{\frac{\mathrm{e}^{-\beta x}K_\nu(\alpha x)}{x^\nu}\int_0^x\mathrm{e}^{\beta t}(t+\mathbb{E}|Z|)t^\nu I_\nu(\alpha t)\,\mathrm{d}t\nonumber\\
&\quad+\frac{\mathrm{e}^{-\beta x}I_\nu(\alpha x)}{x^\nu}\int_x^\infty\mathrm{e}^{\beta t}(t+\mathbb{E}|Z|)t^\nu K_\nu(\alpha t)\,\mathrm{d}t\bigg\}\nonumber
\\ 
&\leq \frac{\|h'\|}{\sigma^2}\bigg\{\sigma^2+\frac{2\sigma^2}{r\sqrt{\theta^2+\sigma^2}}\mathbb{E}|Z|+\frac{\sigma^4}{\theta^2+\sigma^2}+\sqrt{2\pi}|\theta|\sigma\sqrt{r+1}\bigg(1+\frac{\theta^2}{\sigma^2}\bigg)^{\frac{r-1}{2}}\nonumber\\
&\quad+\frac{\sigma^2 A_{r,\theta,\sigma}}{\sqrt{\theta^2+\sigma^2}}\mathbb{E}|Z|\bigg\}\nonumber\\
&\leq\frac{\|h'\|}{\sigma^2}\bigg\{\sigma^2+\frac{2\sigma^2}{r\sqrt{\theta^2+\sigma^2}}\sqrt{r(\sigma^2+2\theta^2)+r^2\theta^2}+\frac{\sigma^4}{\theta^2+\sigma^2}\nonumber\\
\label{tsts}&\quad+\sqrt{2\pi}|\theta|\sigma\sqrt{r+1}\bigg(1+\frac{\theta^2}{\sigma^2}\bigg)^{\frac{r-1}{2}}+\frac{\sigma^2 A_{r,\theta,\sigma}}{\sqrt{\theta^2+\sigma^2}}\sqrt{r(\sigma^2+2\theta^2)+r^2\theta^2}\bigg\}.
\end{align}
By the triangle inequality, 
\begin{align*}\frac{\sqrt{r(\sigma^2+2\theta^2)+r^2\theta^2}}{\sqrt{\theta^2+\sigma^2}}<\frac{\sqrt{2r(\sigma^2+\theta^2)}+r|\theta|}{\sqrt{\theta^2+\sigma^2}}<\sqrt{2r}+r,
\end{align*}
and we also have that $\frac{\sigma^4}{\theta^2+\sigma^2}\leq\sigma^2$. Therefore the bound in (\ref{tsts}) simplifies to
\begin{align*}|f(x)|&\leq \bigg\{1+\frac{2}{r}(\sqrt{2r}+r)+1+\sqrt{2\pi(r+1)}\frac{|\theta|}{\sigma}\bigg(1+\frac{\theta^2}{\sigma^2}\bigg)^{\frac{r-1}{2}}+(\sqrt{2r}+r)A_{r,\theta,\sigma}\bigg\}\|h'\|\\
&=\bigg\{4+\frac{2\sqrt{2}}{\sqrt{r}}+\sqrt{2\pi(r+1)}\frac{|\theta|}{\sigma}\bigg(1+\frac{\theta^2}{\sigma^2}\bigg)^{\frac{r-1}{2}}+(\sqrt{2r}+r)A_{r,\theta,\sigma}\bigg\}\|h'\|.
\end{align*}
It was sufficient to deal with the case $x\geq0$, and so we have proved inequality (\ref{thm1f0}).

We now prove the bound for $\|f'\|$.  For $x\geq0$, the first derivative of $f$ is given by
\begin{align}f'(x)&=-\frac{1}{\sigma^2}\bigg[\frac{\mathrm{d}}{\mathrm{d}x}\bigg(\frac{\mathrm{e}^{-\beta x} K_{\nu}(\alpha x)}{x^{\nu}}\bigg)\bigg] \int_0^x \mathrm{e}^{\beta t} t^{\nu} I_{\nu}(\alpha t) \tilde{h}(t) \,\mathrm{d}t \nonumber \\
\label{fdash}&\quad-\frac{1}{\sigma^2}\bigg[\frac{\mathrm{d}}{\mathrm{d}x}\bigg(\frac{\mathrm{e}^{-\beta x} I_{\nu}(\alpha x)}{x^{\nu}}\bigg)\bigg] \int_x^{\infty} \mathrm{e}^{\beta t} t^{\nu} K_{\nu}(\alpha t)\tilde{h}(t)\,\mathrm{d}t.
\end{align}
Therefore, using inequalities (\ref{jjj1z0}), (\ref{ddd2z0}),  (\ref{ddd2z00}) and (\ref{jjj1z00}), and inequality (\ref{vgabmom}) to bound $\mathbb{E}|Z|$, we have that, for $x\geq0$,
\begin{align}|f'(x)|&=\frac{\|h'\|}{\sigma^2}\bigg\{\bigg|\frac{\mathrm{d}}{\mathrm{d}x}\bigg(\frac{\mathrm{e}^{-\beta x} K_{\nu}(\alpha x)}{x^{\nu}}\bigg)\bigg| \int_0^x \mathrm{e}^{\beta t} (t+\mathbb{E}|Z|)t^{\nu} I_{\nu}(\alpha t) \,\mathrm{d}t \nonumber \\
&\quad+\bigg|\frac{\mathrm{d}}{\mathrm{d}x}\bigg(\frac{\mathrm{e}^{-\beta x} I_{\nu}(\alpha x)}{x^{\nu}}\bigg)\bigg| \int_x^{\infty} \mathrm{e}^{\beta t}(t+\mathbb{E}|Z|) t^{\nu} K_{\nu}(\alpha t)\,\mathrm{d}t\bigg\}\nonumber\\
&\leq\frac{\|h'\|}{\sigma^2}\bigg\{2\sqrt{\theta^2+\sigma^2}+\frac{2}{r}\sqrt{r(\sigma^2+2\theta^2)+r^2\theta^2}+\frac{2\sigma^2}{\sqrt{\theta^2+\sigma^2}}\nonumber\\
\label{rstm}&\quad+2\sqrt{2\pi}|\theta|\sqrt{r+1}\bigg(1+\frac{\theta^2}{\sigma^2}\bigg)^{\frac{r}{2}}+2A_{r,\theta,\sigma}\sqrt{r(\sigma^2+2\theta^2)+r^2\theta^2}\bigg\}.
\end{align}
We bound the upper bound (\ref{rstm}) similarly to how we bounded (\ref{tsts}) in bounding $\|f\|$ to obtain the simpler bound
\begin{align*}|f'(x)|&\leq \frac{\sqrt{\theta^2+\sigma^2}}{\sigma^2}\bigg\{2+\frac{2}{r}(\sqrt{2r}+r)+2+2\sqrt{2\pi(r+1)}\frac{|\theta|}{\sigma}\bigg(1+\frac{\theta^2}{\sigma^2}\bigg)^{\frac{r}{2}}\\
&\quad+2(\sqrt{2r}+r)A_{r,\theta,\sigma}\bigg\}\|h'\|\\
&=\frac{\sqrt{\theta^2+\sigma^2}}{\sigma^2}C_{r,\theta,\sigma}\|h'\|.
\end{align*}
Again, it suffices to consider $x\geq0$, so we have proved inequality (\ref{thm1f1}).

To bound $\|f''\|$, we use the iterative technique of \cite{dgv15}. %A detailed account of the technique for general classes of Stein equations is given on pages 4--5 of \cite{dgv15}, and Assumption 2.1 and Remark 2.2 of that work show that the VG Stein equation (\ref{377}) falls within that framework.
A detailed account of the technique is given on pages 4--5 of \cite{dgv15}, and it was noted in that work that the technique is applicable to the VG Stein equation (\ref{377}) (see their Assumption 2.1 and Remark 2.2).
%A detailed account of the technique is given on pages 4--5 of \cite{dgv15}, and Assumption 2.1 and Remark 2.2 of that work show that the iterative technique can be applied in the case of the VG Stein equation (\ref{377}).
   We differentiate both sides of the $\mathrm{VG}(r,\theta,\sigma,0)$ Stein equation (\ref{377}) and rearrange to obtain
\begin{equation}\label{iteqn}\sigma^2 xf^{(3)}(x)+(\sigma^2(r+1)+2\theta x)f''(x)+((r+1)\theta -x)f'(x)=h'(x)+f(x)+\theta f'(x).
\end{equation}
We recognise (\ref{iteqn}) as the $\mathrm{VG}(r+1,\theta,\sigma,0)$ Stein equation, applied to $f'$, with test function $h'(x)+f(x)+\theta f'(x)$. Indeed, (\ref{iteqn}) can be written compactly as
$L_{r+1,\theta,\sigma,0}f'(x)=h'(x)+f(x)+\theta f'(x),$
where $L_{r+1,\theta,\sigma,0}$ is the $\mathrm{VG}(r+1,\theta,\sigma,0)$ Stein operator. Here, the test function $h'(x)+f(x)+\theta f'(x)$ has mean zero with respect to the random variable $Y\sim\mathrm{VG}(r+1,\theta,\sigma,0)$.  We will make use of this property when we later apply inequality (\ref{vgsolnunibound1}).  As $h$ is Lipschitz, it follows from inequalities (\ref{thm1f0}) and (\ref{thm1f1}) that $\mathbb{E}|h'(Y)+f(Y)+\theta f'(Y)|<\infty$. In particular, because (\ref{iteqn}) is the $\mathrm{VG}(r+1,\theta,\sigma,0)$ Stein equation applied to $f'$, it follows that $\mathbb{E}[L_{r+1,\theta,\sigma,0}f'(Y)]=0$, and so $\mathbb{E}[h'(Y)+f(Y)+\theta f'(Y)]=0$.    
%\rg{Crucially, for Lipschitz $h$, $\mathbb{E}[h'(Z)+f(Z)]=0$ for $Z\sim\mathrm{SVG}(r+1,1,0)$.  To see this, we recall from Lemma 3.1 of \cite{gaunt vg} that if $g:\mathbb{R}\rightarrow\mathbb{R}$ is twice differentiable and is such that $\mathbb{E}|Zg''(Z)|$, $\mathbb{E}|g'(Z)|$ and $\mathbb{E}|Zg(Z)|$ are all finite, then $\mathbb{E}[Zg''(Z)+(r+1)Zg'(Z)-Zg(Z)]=0$.  Now, as $h$ is Lipschitz, we have from inequalities (\ref{thm1f0})--(\ref{thm1f2}) that there exist constants $C_{0,r},C_{1,r},C_{2,r}<\infty$ such that $\|f^{(i)}\|\leq C_{i,r}\|h'\|$, $i=0,1,2$, where $f^{(0)}\equiv f$. Since $E|Z|<\infty$, it thus follows that $\mathbb{E}|Zf'(Z)|<\infty$, $\mathbb{E}|f''(Z)|<\infty$, and from (\ref{iteqn}),
%\begin{equation*}\mathbb{E}|Zf^{(3)}(Z)|<\mathbb{E}|h'(Z)|+\mathbb{E}|f(Z)|+(r+1)\mathbb{E}|f''(Z)|+\mathbb{E}|Zf'(Z)|<\infty.
%\end{equation*}
%Therefore the assumptions of Lemma 3.1 of \cite{gaunt vg} are satisfied, and we obtain
%\begin{equation*}0=\mathbb{E}[Zf^{(3)}(Z)+(r+1)f''(Z)-Zf'(Z)]=\mathbb{E}[h'(Z)+f(Z)],
%\end{equation*}
%as required.}  
An application of inequality  (\ref{vgsolnunibound1}), with $r$ replaced by $r+1$ and test function $h'(x)+f(x)+\theta f'(x)$, now gives that
\begin{align}\|f''\|&=\frac{1}{\sigma^2}\bigg(\frac{2}{r+1}+A_{r+1,\theta,\sigma}\bigg)\|h'(x)+f(x)+\theta f'(x)\| \nonumber\\
\label{dfghy}&\leq\frac{1}{\sigma^2}\bigg(\frac{2}{r+1}+A_{r+1,\theta,\sigma}\bigg)\big(\|h'\|+\|f\|+|\theta|\|f'\|\big).
\end{align}
To obtain the bound (\ref{thm1f2}) for  $\|f''\|$, we use (\ref{thm1f0}) and (\ref{thm1f1}) to bound $\|f\|$ and $\|f'\|$, respectively, and simplify to obtain
\begin{align}\|h'\|+\|f\|+|\theta|\|f'\|&\leq\bigg\{1+\bigg(1+\frac{|\theta|\sqrt{\theta^2+\sigma^2}}{\sigma^2}\bigg)C_{r,\theta,\sigma}\bigg\}\|h'\|\nonumber\\
\label{frfr5}&\leq\bigg\{1+\bigg(2+\frac{\theta^2}{\sigma^2}\bigg)C_{r,\theta,\sigma}\bigg\}\|h'\|,
\end{align} 
where we used the inequality $\frac{|\theta|\sqrt{\theta^2+\sigma^2}}{\sigma^2}<1+\frac{\theta^2}{\sigma^2}$, since $|\theta|<\sqrt{\theta^2+\sigma^2}$. Combining inequalities (\ref{dfghy}) and (\ref{frfr5}) gives us the bound (\ref{thm1f2}), as required.

Suppose now that $h:\mathbb{R}\rightarrow\mathbb{R}$ is bounded and measurable. We now prove the bounds (\ref{thmone})--(\ref{thmthree}).  Using the integral inequalities (\ref{jjj1b0}) and (\ref{ddd2}), we obtain, for $x\geq0$,
\begin{align*}|xf(x)|&=\bigg|\frac{\mathrm{e}^{-\beta x}K_\nu(\alpha x)}{\sigma^2 x^{\nu-1}}\int_0^x \mathrm{e}^{\beta t}t^\nu I_\nu(\alpha t)\tilde{h}(t)\,\mathrm{d}t+\frac{\mathrm{e}^{-\beta x}I_\nu(\alpha x)}{\sigma^2 x^{\nu-1}}\int_x^\infty \mathrm{e}^{\beta t}t^\nu K_\nu(\alpha t)\tilde{h}(t)\,\mathrm{d}t\bigg| \\
&\leq\frac{\|\tilde{h}\|}{\sigma^2}\bigg\{\frac{\mathrm{e}^{-\beta x}K_\nu(\alpha x)}{x^{\nu-1}}\int_0^x \mathrm{e}^{\beta t}t^\nu I_\nu(\alpha t)\,\mathrm{d}t+\frac{\mathrm{e}^{-\beta x}I_\nu(\alpha x)}{x^{\nu-1}}\int_x^\infty \mathrm{e}^{\beta t}t^\nu K_\nu(\alpha t)\,\mathrm{d}t\bigg\} \\
&\leq \frac{\|\tilde{h}\|}{\sigma^2}\bigg\{\sigma^2\bigg(1+\frac{6}{r}\bigg)+\sigma^2B_{r,\theta,\sigma}\bigg\}=\|\tilde{h}\|\bigg(1+\frac{6}{r}+B_{r,\theta,\sigma}\bigg).
\end{align*}
Using the formula (\ref{fdash}) for the first derivative of $f$, followed by an application of the integral inequalities (\ref{jjj1z}) and (\ref{ddd2z}), gives that for $x\geq0$,
\begin{align*}|xf'(x)|&=\bigg|\frac{x}{\sigma^2}\bigg[\frac{\mathrm{d}}{\mathrm{d}x}\bigg(\frac{\mathrm{e}^{-\beta x}K_{\nu}(\alpha x)}{x^{\nu}}\bigg)\bigg]\int_0^x \mathrm{e}^{\beta t}t^\nu I_\nu(\alpha t)\tilde{h}(t)\,\mathrm{d}t\\
&\quad+\frac{x}{\sigma^2}\bigg[\frac{\mathrm{d}}{\mathrm{d}x}\bigg(\frac{\mathrm{e}^{-\beta x}I_{\nu}(\alpha x)}{x^{\nu}}\bigg)\bigg]\int_x^\infty \mathrm{e}^{\beta t}t^\nu K_\nu(\alpha t)\tilde{h}(t)\,\mathrm{d}t\bigg| 
\end{align*}
\begin{align*}
&\leq \frac{\|\tilde{h}\|}{\sigma^2}\bigg\{x\bigg|\frac{\mathrm{d}}{\mathrm{d}x}\bigg(\frac{\mathrm{e}^{-\beta x}K_{\nu}(\alpha x)}{x^{\nu}}\bigg)\bigg|\int_0^x \mathrm{e}^{\beta t}t^\nu I_\nu(\alpha t)\,\mathrm{d}t\\
&\quad+x\bigg|\frac{\mathrm{d}}{\mathrm{d}x}\bigg(\frac{\mathrm{e}^{-\beta x}I_{\nu}(\alpha x)}{x^{\nu}}\bigg)\bigg|\int_x^\infty \mathrm{e}^{\beta t}t^\nu K_\nu(\alpha t)\,\mathrm{d}t\bigg\} \\
&\leq \frac{\|\tilde{h}\|}{\sigma^2}\bigg\{2\bigg(1+\frac{6}{r}\bigg)\sqrt{\theta^2+\sigma^2}+2\sqrt{\theta^2+\sigma^2}B_{r,\theta,\sigma}\bigg\}\\
%&=\frac{2}{\sigma}\bigg(1+\frac{\theta^2}{\sigma^2}\bigg)^{\frac{1}{2}}\bigg(1+\frac{6}{r}+B_{r,\theta,\sigma}\bigg)\|\tilde{h}\|.\\
&=\frac{2\sqrt{\theta^2+\sigma^2}}{\sigma^2}\bigg(1+\frac{6}{r}+B_{r,\theta,\sigma}\bigg)\|\tilde{h}\|.
\end{align*}
Again, it was sufficient to deal with the $x\geq0$ case, and so we have proved (\ref{thmone}) and (\ref{thmtwo}). We now obtain the bound for $\|xf''(x)\|$, and we begin by rearranging the $\mathrm{VG}(r,\theta,\sigma,0)$ Stein equation and using the triangle inequality to obtain that, for $x\in\mathbb{R}$, 
\begin{align*}|xf''(x)|&=\frac{1}{\sigma^2}|\tilde{h}(x)-(\sigma^2r+2\theta x)f'(x)-(r\theta -x)f(x)|\\
&\leq \frac{1}{\sigma^2}\|\tilde{h}\|+r\|f'\|+\frac{2|\theta|}{\sigma^2}\|xf'(x)\|+\frac{r|\theta|}{\sigma^2}\|f\|+\frac{1}{\sigma^2}\|xf(x)\|.
\end{align*}
We now use (\ref{vgsolnunibound1}) to bound $\|f'\|$, (\ref{thmtwo}) to bound $\|xf'(x)\|$, (\ref{vgsolnunibound}) to bound $\|f\|$, and (\ref{thmone}) to bound $\|xf(x)\|$, which gives us the bound
\begin{align*}\|xf''(x)\|&\leq \bigg\{\frac{1}{\sigma^2}+r\cdot\frac{1}{\sigma^2}\bigg(\frac{2}{r}+A_{r,\theta,\sigma}\bigg)+\frac{2|\theta|}{\sigma^2}\cdot\frac{2\sqrt{\theta^2+\sigma^2}}{\sigma^2}\bigg(1+\frac{6}{r}+B_{r,\theta,\sigma}\bigg)\\
&\quad+\frac{r|\theta|}{\sigma^2}\cdot\frac{1}{\sqrt{\theta^2+\sigma^2}}\bigg(\frac{2}{r}+A_{r,\theta,\sigma}\bigg)+\frac{1}{\sigma^2}\cdot\bigg(1+\frac{6}{r}+B_{r,\theta,\sigma}\bigg)\bigg\}\|\tilde{h}\|\\
&\leq \bigg\{\frac{1}{\sigma^2}+r\cdot\frac{1}{\sigma^2}\bigg(\frac{2}{r}+A_{r,\theta,\sigma}\bigg)+\frac{4(\theta^2+\sigma^2)}{\sigma^4}\bigg(1+\frac{6}{r}+B_{r,\theta,\sigma}\bigg)\\
&\quad+\frac{r}{\sigma^2}\bigg(\frac{2}{r}+A_{r,\theta,\sigma}\bigg)+\frac{1}{\sigma^2}\cdot\bigg(1+\frac{6}{r}+B_{r,\theta,\sigma}\bigg)\bigg\}\|\tilde{h}\|\\
&=\frac{1}{\sigma^2}\bigg\{5+2rA_{r,\theta,\sigma}+\bigg(5+\frac{4\theta^2}{\sigma^2}\bigg)\bigg(1+\frac{6}{r}+B_{r,\theta,\sigma}\bigg)\bigg\}\|\tilde{h}\|.
\end{align*}

%The bounds for $\|xf'(x)\|$, $\|xf''(x)\|$ and $\|xf^{(3)}(x)\|$ for Lipschitz $h$ can be readily obtained through an application of the iterative technique with (\ref{thmone}), (\ref{thmtwo}) and (\ref{thmthree}) as the base case, respectively:

Finally, we bound $\|xf'(x)\|$, $\|xf''(x)\|$ and $\|xf^{(3)}(x)\|$ for Lipschitz $h$. We do so through a similar application of the iterative technique of \cite{dgv15} to the one we used to establish inequality (\ref{thm1f2}).  The setting is the same in that (\ref{iteqn}) is the $\mathrm{VG}(r+1,\theta,\sigma,0)$ Stein equation, applied to $f'$, with the test function $h'(x)+f(x)+\theta f'(x)$ having zero mean with respect to the $\mathrm{VG}(r+1,\theta,\sigma,0)$ measure.   We apply inequalities (\ref{thmone}), (\ref{thmtwo}) and (\ref{thmthree}), respectively, with $r$ replaced by $r+1$ and test function $h'(x)+f(x)+\theta f'(x)$, to obtain the bounds
\begin{eqnarray*}\|xf'(x)\|&\leq&\bigg(1+\frac{6}{r+1}+B_{r+1,\theta,\sigma}\bigg)\|h'(x)+f(x)+\theta f'(x)\|\\
&\leq&\bigg(1+\frac{6}{r+1}+B_{r+1,\theta,\sigma}\bigg)\big(\|h'\|+\|f\|+|\theta|\|f'\|\big),
\end{eqnarray*}
\begin{eqnarray*}
\|xf''(x)\|&\leq&\frac{2\sqrt{\theta^2+\sigma^2}}{\sigma^2}\bigg(1+\frac{6}{r+1}+B_{r+1,\theta,\sigma}\bigg)\big(\|h'\|+\|f\|+|\theta|\|f'\|\big),\\
\|xf^{(3)}(x)\|&\leq& \frac{1}{\sigma^2}\bigg\{5+2(r+1)A_{r+1,\theta,\sigma}+\bigg(5+\frac{4\theta^2}{\sigma^2}\bigg)\bigg(1+\frac{6}{r+1}+B_{r+1,\theta,\sigma}\bigg)\bigg\}\\
&&\times\big(\|h'\|+\|f\|+|\theta|\|f'\|\big). 
\end{eqnarray*}
Using inequality (\ref{frfr5}) to bound $\|h'\|+\|f\|+|\theta|\|f'\|$ then yields the bounds (\ref{thmfour})--(\ref{thmsix}).  This completes the proof. \hfill $\Box$

\vspace{3mm}

\noindent{\emph{Proof of Corollary \ref{cortothm1}.}} As in the proof of Theorem \ref{thm1}, we set $\mu=0$. We obtain the bound by using a similar implementation of the iterative technique of \cite{dgv15} to those used in the proof of Theorem \ref{thm1}. Recall that (\ref{iteqn}) is the $\mathrm{VG}(r+1,\theta,\sigma,0)$ Stein equation, applied to $f'$, with the test function $h'(x)+f(x)+\theta f'(x)$ having zero mean with respect to the $\mathrm{VG}(r+1,\theta,\sigma,0)$ measure.   By applying inequality (\ref{thm1f2}) with $r$ replaced by $r+1$ and test function $h'(x)+f(x)+\theta f'(x)$, we obtain the bound
\begin{align*}\|f^{(3)}\|&\leq \frac{1}{\sigma^2}\bigg(\frac{2}{r+2}+A_{r+2,\theta,\sigma}\bigg)\bigg\{1+\bigg(2+\frac{\theta^2}{\sigma^2}\bigg)C_{r+1,\theta,\sigma}\bigg\}\|h''(x)+f'(x)+\theta f''(x)\|\\
&\leq \frac{1}{\sigma^2}\bigg(\frac{2}{r+2}+A_{r+2,\theta,\sigma}\bigg)\bigg\{1+\bigg(2+\frac{\theta^2}{\sigma^2}\bigg)C_{r+1,\theta,\sigma}\bigg\}\big(\|h''\|+\|f'\|+|\theta|\|f''\|\big).
\end{align*}
Bounding $\|f'\|$ and $\|f''\|$ using inequalities (\ref{thm1f1}) and (\ref{thm1f2}), respectively, then yields the desired bound on $\|f^{(3)}\|$. \hfill $\Box$

\begin{remark}\label{appbr}Recall the change of parameters (\ref{parameter}).  The following bound (\ref{dgv1}) was given on p$.$ 24 of \cite{dgv15}, and the bound (\ref{dgv2}) is a slight improvement on one given on p$.$ 24 of \cite{dgv15}:
\begin{eqnarray}\label{dgv1}\|f\|&\leq& \frac{\|\tilde{h}\|}{\sigma^2\alpha}\bigg(\frac{2}{2\nu+1}+M_{\nu,\gamma}\bigg), \\
\label{dgv2}\|f'\|&\leq& \frac{\|\tilde{h}\|}{\sigma^2}\bigg(\frac{2}{2\nu+1}+M_{\nu,\gamma}\bigg),
\end{eqnarray}
where $M_{\nu,\gamma}$ is defined in (\ref{mdefn}).
The improvement in inequality (\ref{dgv2}) for $\|f'\|$ comes from using the integral inequality (\ref{jjj1z00}), which improves on the analogous integral inequality that was used by \cite{dgv15} in deriving their bound for $\|f'\|$. The bounds for $\|f\|$ and $\|f'\|$ of \cite{dgv15} were translated into the $\mathrm{VG}(r,\theta,\sigma,\mu)$ parametrisation on p$.$ 17 of \cite{dgv15} at the cost of two typos. By using the improved integral inequality (\ref{jjj1z00}), we have been able to fix one of the typos made by \cite{dgv15} (this concerns the factor $\frac{2}{r}$ in the bound (\ref{vgsolnunibound1})).  We correct the other typo by correctly bounding $M_{\nu,\gamma}<A_{r,\theta,\sigma}$ (see (\ref{ambamb})), which leads to different, corrected, bounds to those stated by \cite{dgv15}. In obtaining the inequality $M_{\nu,\gamma}<A_{r,\theta,\sigma}$ in Appendix \ref{appb}, we also obtained a slight simplification on the presentation given in \cite{dgv15} by using the upper bound in (\ref{oct8}) to bound a ratio of gamma functions by a power function.
\end{remark}

It is a natural question to ask whether bounds of the form $\|f''\|\leq M_{r,\theta,\sigma}\|\tilde{h}\|$ and $\|f^{(3)}\|\leq M_{r,\theta,\sigma}\|h'\|$, where $M_{r,\theta,\sigma}>0$ is a constant not involving $x$, could be obtained that hold for all bounded and measurable $h:\mathbb{R}\rightarrow\mathbb{R}$, and all Lipschitz $h:\mathbb{R}\rightarrow\mathbb{R}$, respectively. We will show that this is not possible through the following two propositions, which are proved in Section \ref{sec6}. Analogous results for the $\theta=0$ case are given in \cite{gaunt vgii}; our propositions show that no such bounds are attainable for any possible choice of parameter values in the four parameter VG class.  We also refer the reader to \cite{eden2} for similar results concerning solutions of Stein equations for a wide class of distributions.

%It is a natural question to ask whether similar upper bounds (as in constant multiples of $\|\tilde{h}\|$ or $\|h'\|$) to those given in (\ref{}) and (\ref{}) and Theorem \ref{thm1} can be obtained that involve higher order derivatives of the solution of the VG Stein equation. We will show that this is not possible through the following two propositions, which are proved in Section \ref{sec6}.

%Denote by $f_z$ the solution to the $\mathrm{VG}(r,\theta,\sigma,\mu)$ Stein equation (\ref{377}) with test function $h(x)=\mathbf{1}(x\leq z)$. This is the form of the VG Stein equation that is appropriate for deriving error bounds in the Kolmogorov distance.  Bounds for $f_z$ can be deduced from those of Theorem \ref{thm1} by using that $\|\tilde{h}_z\|\leq1$.  

%It is natural to ask whether, for all $z\in\mathbb{R}$, a bound of the form $\|f_z''\|\leq C_{r,\sigma}$ could be obtained for the solution $f_z$.  The following proposition, which is proved in Section \ref{sec6}, shows that this is not possible.

\begin{proposition}\label{disclem}Denote by $f_z$ the solution to the $\mathrm{VG}(r,\theta,\sigma,\mu)$ Stein equation (\ref{377}) with test function $h(x)=\mathbf{1}(x\leq z)$. Then $f_z'(x)$ is discontinuous at $x=\mu$.
\end{proposition}

\begin{proposition}\label{ptpt}Let $f$ denote the solution to the $\mathrm{VG}(r,\theta,\sigma,\mu)$ Stein equation with Lipschitz test function $h:\mathbb{R}\rightarrow\mathbb{R}$. Then there does not exist a positive constant $M_{r,\theta,\sigma}$ such that the bound $\|f^{(3)}\|\leq M_{r,\theta,\sigma}\|h'\|$ holds for all Lipschitz $h:\mathbb{R}\rightarrow\mathbb{R}$.
\end{proposition}

%Similarly, one may ask whether a bound of the form $\|f^{(3)}\|\leq C_{r,\sigma}\|h'\|$ could be obtained for all Lipschitz $h:\mathbb{R}\rightarrow\mathbb{R}$.  The following proposition, which is proved in Section \ref{sec6}, again shows this is not possible (see \cite{eden2} for similar results that apply to solutions of Stein equations for a wide class of distributions).  Our approach differs from that of Proposition \ref{disclem} in that we do not find a Lipschitz test function $h$ for which $f''$ has a discontinuity.  This would be more tedious to establish for $f''$ than for $f'$ and instead we consider a highly oscillating test function and perform an asymptotic analysis.    

We end this section by stating the following proposition, which relates the Kolmogorov and Wasserstein distances between a general distribution and a VG distribution. The proof of Proposition \ref{prop1} is postponed to Section \ref{sec6}. This is a useful result, because, for continuous target distributions, it is typically easier to obtain Wasserstein distance bounds via Stein's method than Kolmogorov distance bounds.  This is indeed the case in our application to the Malliavin-Stein method for VG approximation in Section \ref{sec4}.

\begin{proposition}\label{prop1}Let $Z\sim\mathrm{VG}(r,\theta,\sigma,\mu)$, where $r>0$, $\theta\in\mathbb{R}$, $\sigma>0$ and $\mu\in\mathbb{R}$.  Let $p_{r,\sigma,\theta}(x)$ denote the density (\ref{vgdefn}) with $\mu=0$.  Then, for any random variable $W$:

%(i) If $1<r\leq2$,
%\begin{equation}\label{pronf1}d_{\mathrm{K}}(W,Z)\leq\sqrt{\frac{1}{\sigma\sqrt{\pi}(1+\theta^2/\sigma^2)^{\frac{r-1}{2}}}\frac{\Gamma\big(\frac{r-1}{2}\big)}{\Gamma\big(\frac{r}{2}\big)}d_{\mathrm{W}}(W,Z)}.
%\end{equation}

\vspace{2mm}

\noindent{(i)} If $r>1$,
\begin{equation*}d_{\mathrm{K}}(W,Z)\leq D_{r,\sigma,\theta}\sqrt{d_{\mathrm{W}}(W,Z)},
\end{equation*}
where $D_{r,\sigma,\theta}=\sup_{x\in\mathbb{R}}\sqrt{2p_{r,\sigma,\theta}(x)}$. When $1<r\leq2$ we have
\begin{equation*}D_{r,\sigma,\theta}=\sqrt{\frac{\Gamma(\frac{r-1}{2})}{\sigma\sqrt{\pi} \Gamma(\frac{r}{2})}\bigg(\frac{\sigma^2}{\theta^2+\sigma^2}\bigg)^{\frac{r-1}{2}} },
\end{equation*}
and when $r>2$ we have
\begin{equation*}D_{r,\sigma,\theta}=\sqrt{2p_{r,\sigma,\theta}(x^*\mathrm{sgn}(\theta))}\leq \bigg\{\frac{\Gamma(\frac{r-1}{2})}{\sigma\sqrt{\pi} \Gamma(\frac{r}{2})}\bigg(\frac{\sigma^2}{\theta^2+\sigma^2}\bigg)^{\frac{r-1}{2}} \mathrm{e}^{\frac{\theta^2}{\sigma^2} (r-2)}\bigg\}^{\frac{1}{2}},
\end{equation*}
where $x^*$ is the unique positive solution of (\ref{xstar}). In the case $r>3$, $\theta\not=0$, a more accurate bound on $D_{r,\theta,\sigma}$ can be obtained by bounding $p_{r,\sigma,\theta}(x^*\mathrm{sgn}(\theta))$ using inequality (\ref{vgpdfineq}). 
%and $A_{r,\theta,\sigma}$ is defined as in Proposition \ref{prop2mode}.

\vspace{2mm}

\noindent{(ii)} Let $r=1$. Suppose that $\frac{\theta^2+\sigma^2}{\sigma^3} d_{\mathrm{W}}(W,Z)<0.755$.  Then
\begin{equation}\label{pronf2}d_{\mathrm{K}}(W,Z)\leq \bigg\{5+\log\bigg(\frac{6}{\pi}\bigg)+\log\bigg(\frac{\sigma^3}{(\theta^2+\sigma^2)d_{\mathrm{W}}(W,Z)}\bigg)\bigg\}\sqrt{\frac{d_{\mathrm{W}}(W,Z)}{6\pi\sigma}}.
\end{equation}

\noindent{(iii)} If $0<r<1$,
\begin{equation}\label{pronf3}d_{\mathrm{K}}(W,Z)\leq 2\bigg(\frac{\Gamma\big(\frac{1-r}{2}\big)}{\sqrt{\pi}2^{r-1} \Gamma\big(\frac{r}{2}\big)}\bigg)^{\frac{1}{r+1}}\big(\sigma^{-1}d_{\mathrm{W}}(W,Z)\big)^{\frac{r}{r+1}}.
\end{equation}
\end{proposition}

\begin{remark}(i) The assumption that $\frac{\theta^2+\sigma^2}{\sigma^3} d_{\mathrm{W}}(W,Z)<0.755$ is quite mild.  Indeed, if $\theta=0$, then with $d_{\mathrm{W}}(W,Z)/\sigma=0.755$ we see that the upper bound in (\ref{pronf2}) is equal to 1.186, and thus uninformative. It is possible to increase the range of validity of inequality (\ref{pronf2}) (that is increase the numerical constant that $\frac{\theta^2+\sigma^2}{\sigma^3} d_{\mathrm{W}}(W,Z)$ is bounded above by beyond 0.755) at the expense of larger numerical constants in the upper bound.  This can be done by making a minor modification to derivation of inequality (\ref{pronf2}) by applying the more general part (ii) of Lemma \ref{beslem} with $c>3$, rather than part (iii) of that lemma with $c=3$. We proceeded as we did to simplify the statement and proof of part (ii) of Proposition \ref{prop1}.
% If $\frac{\theta^2+\sigma^2}{\sigma^3} d_{\mathrm{W}}(W,Z)=0.755$,  the upper bound in (\ref{pronf2}) is equal to 1.186, and thus uninformative.

\vspace{2mm}

%\noindent{(ii)} Note that the bound (\ref{pronf3}) does not involve $\theta$.  
%Setting $\theta=0$ in the bounds (\ref{pronf1}) and (\ref{pronf2}) yields 

%\vspace{2mm}

\noindent{(ii)} An analogue of Proposition \ref{prop1} was given by \cite[Proposition 4.1]{gaunt vgii} for the $\theta=0$ case.  Our bounds for the general $\theta\in\mathbb{R}$ case take the same functional form in terms of dependence on $d_\mathrm{W}(W,Z)$ as those of \cite{gaunt vgii}. In fact, the bound (\ref{pronf3}), which does not involve $\theta$, is exactly the same as that of  \cite{gaunt vgii} for $0<r<1$ in the $\theta=0$ case. In general, we expect our inequalities to yield suboptimal order Kolmogorov distance bounds.  Indeed, an example has been given in the $\theta=0$ case in which bounds for each of the cases $r>1$, $r=1$ and $0<r<1$ are seen to suboptimal; see Remark 5.2 of \cite{gaunt vgii}.
\end{remark}

\section{Application to the Malliavin-Stein method for variance-gamma approximation}\label{sec4}

%In recent years, one of the most significant applications of Stein's method has been to Gaussian analysis on Wiener space.  This body of research was initiated by \cite{np09}, in which Stein's method and Malliavin calculus are combined to derive a quantitative ``fourth moment" theorem for the normal approximation of a sequence of random variables living in a fixed Wiener chaos.

%\subsection{A quantitative sixth moment theorem}

In this section, we obtain explicit constants in some of the main results of the paper \cite{eichelsbacher} (see Theorem \ref{etthm} and Corollary \ref{cor5.4} below), which extended the Malliavin-Stein method to the VG distribution. In doing so, we fix a technical issue in that the Wasserstein distance bounds stated in \cite{eichelsbacher} had only been proven in the weaker bounded Wasserstein distance. This is because at the time of \cite{eichelsbacher} the only available bounds for the solution of the VG Stein equation \cite{gaunt thesis, gaunt vg} had a dependence on the test function $h$ that meant that this was the best that could be attained. We also give an illustrative example of the applicability of the general bound in Corollary \ref{cor5.4} by obtaining bounds on the rate of convergence in a recent result of \cite{bt17} concerning the generalized Rosenblatt process at extreme critical exponent.  

%In a recent work \cite{eichelsbacher}, the Malliavin-Stein method was extended to the VG distribution. Here, we obtain explicit constants in some of the main results (in the SVG case) of \cite{eichelsbacher}, these being six moment theorems for the SVG approximation of double Wiener-It\^{o} integrals. Our results also fix a technical issue in that the Wasserstein distance bounds stated in \cite{eichelsbacher} had only been proven in the weaker bounded Wasserstein distance (at the time of \cite{eichelsbacher} the bounds for the solution of the Stein equation in the literature \cite{gaunt thesis, gaunt vg} had a dependence on the test function $h$ such that this was the best that could be achieved).

%Let $\mathfrak{H}$ be a real separable Hilbert space with inner product $\langle,\rangle_{\mathfrak{H}}$. Let $X=\{X(h)\,:\,h\in\mathfrak{H}\}$ be an isonormal Gaussian process, defined on a probability space $(\Omega,\mathscr{F},\mathbb{P})$, where $\mathscr{F}$ is the $\sigma$-algebra generated by $X$. This means that $X$ is a family of centered, jointly Gaussian random variables with $\mathbb{E}[X(g)X(h)]=\langle{g,h\rangle_{\mathfrak{H}}$. For 

We first introduce some notation; see the book \cite{np12} for further details.  We write $\mathbb{D}^{p,q}$ to denote the Banach space of all functions in $L^q(\gamma)$, where $\gamma$ is the standard Gaussian measure, whose Malliavin derivatives up to order $p$ belong to $L^q(\gamma)$.  The class of infinitely many times Malliavin differentiable random variables is denoted by $\mathbb{D}^\infty$. For a random variable $F\in\mathbb{D}^\infty$, we iteratively define the gamma operators $\Gamma_j$ \cite{np10} by $\Gamma_1(F)=F$ and, for $j\geq2$,
\[\Gamma_j(F)=\langle DF, -DL^{-1}\Gamma_{j-1}(F)\rangle_{\mathfrak{H}}.\]
Here $\mathfrak{H}$ is a real separable Hilbert space, $D$ is the Malliavin derivative, and $L^{-1}$ is the pseudo-inverse of the infinitesimal generator of the Ornstein-Uhlenbeck semi-group. 
%For an integer $q\geq1$, we let $\mathfrak{H}^{\odot q}$ denote the $q$-th symmetric tensor product of $\mathfrak{H}$. 
Let $\mathfrak{H}^{\odot 2}$ denote the second symmetric tensor product of $\mathfrak{H}$.
  For $f\in \mathfrak{H}^{\odot 2}$, the double Wiener-It\^{o} integral is denoted by $I_2(f)$ (see \cite[Definition 2.7.1]{np12}). Some of the most important properties of multiple  Wiener-It\^{o} integrals are given in Section 2.7 of \cite{np12}. Double Wiener-It\^{o} integrals also have several attractive properties and representations that are not shared by higher order multiple Wiener-It\^{o} integrals; see \cite[Section 2.7.4]{np12}.  Recall that we write $\mathrm{VG}_c(r,\theta,\sigma)$ for $\mathrm{VG}(r,\theta,\sigma,-r\theta)$.

%\begin{equation}\label{fmal}F=\sum_{q=0}^\infty I_q(f_q).
%\end{equation}
%$L=\sum_{q=0}^\infty -qJ_q$, where $J_q(F)=I_q(f_q)$ for any $F$ as in (\ref{fmal}).

%We introduce the so-called $\Gamma$-operators $\Gamma_j$.  For $F\in\mathbb{D}^\infty$, we define $\Gamma_1(F)=F$ and, for every $j\geq2$,
%\[\Gamma_j(F)=\langle DF, -DL^{-1}\Gamma_{j-1}(F)\rangle_{\mathfrak{H}}.\]

\begin{theorem}\label{etthm}Let $F\in\mathbb{D}^{3,8}$ and suppose that $\Gamma_3(F)$ is square-integrable and $\mathbb{E}F=0$. Then, for $Z\sim \mathrm{VG}_c(r,\theta,\sigma)$,
\begin{equation}\label{etqu}d_{\mathrm{W}}(F,Z)\leq C_1\big(\mathbb{E}[(\sigma^2 (F+r\theta)+2\theta\Gamma_2(F)-\Gamma_3(F))^2]\big)^{\frac{1}{2}} +C_2\mathbb{E}|r(\sigma^2+2\theta^2)-\mathbb{E}[\Gamma_2(F)]|,
\end{equation}
where
\begin{align*}C_1&=\frac{1}{\sigma^2}\bigg(\frac{2}{r+1}+A_{r+1,\theta,\sigma}\bigg)\bigg\{1+\bigg(2+\frac{\theta^2}{\sigma^2}\bigg)C_{r,\theta,\sigma}\bigg\}, \quad C_2=\frac{\sqrt{\theta^2+\sigma^2}}{\sigma^2}C_{r,\theta,\sigma},
\end{align*}
with $C_{r,\theta,\sigma}$ defined as in (\ref{cdefn}).
\end{theorem}

\begin{proof} It was shown in the proof of Theorem 4.1 of \cite{eichelsbacher} that, for functions $f:\mathbb{R}\rightarrow\mathbb{R}$ that are twice differentiable with bounded first and second derivative,
\begin{align}&\big|\mathbb{E}\big[\sigma^2Ff''(F)+\sigma^2 rf'(F)-Ff(F)\big]\big|\nonumber\\
&\quad=\big|\mathbb{E}\big[f''(F)(\sigma^2 (F+r\theta)+2\theta\Gamma_2(F)-\Gamma_3(F))+f'(F)(r\sigma^2+2r\theta^2-\mathbb{E}[\Gamma_2(F)])\big]\big|\nonumber  \\
&\quad\leq \|f''\|\mathbb{E}|\sigma^2 (F+r\theta)+2\theta\Gamma_2(F)-\Gamma_3(F)|+\|f'\|\mathbb{E}|r(\sigma^2+2\theta^2)-\mathbb{E}[\Gamma_2(F)]|\nonumber\\
&\quad\leq \|f''\|\big(\mathbb{E}[(\sigma^2 (F+r\theta)+2\theta\Gamma_2(F)-\Gamma_3(F))^2]\big)^{\frac{1}{2}}+\|f'\|\mathbb{E}|r(\sigma^2+2\theta^2)-\mathbb{E}[\Gamma_2(F)]|,\nonumber
\end{align} 
where a justification of the application of the Cauchy-Schwarz inequality in the final step is given in \cite{eichelsbacher}.
We know from Theorem \ref{thm1} that, for $h\in\mathcal{H}_{\mathrm{W}}$,  the solution $f$ of the $\mathrm{VG}_c(r,\theta,\sigma)$ Stein equation satisfies the conditions of being twice differentiable with bounded first and second derivatives. We can bound $\|f''\|$ and $\|f'\|$ using the estimates (\ref{thm1f2}) and (\ref{thm1f1}) of Theorem \ref{thm1} (with $\|h'\|=1$), which gives (\ref{etqu}). 
\end{proof}

\begin{corollary}\label{cor5.4} 
Consider the sequence $(F_n=I_2(f_n)\,:\,n\geq1)$ with $f_n\in \mathfrak{H}^{\odot 2}$, $n\geq1$.
%Let $F_n=I_2(f_n)$ with $f_n\in \mathfrak{H}^{\odot 2}$, $n\geq1$. 
%, and suppose that $\mathbb{E}[F_n^2]=r(\sigma^2+2\theta^2)$.  
Let $Z\sim\mathrm{VG}_c(r,\theta,\sigma)$ and write $\tilde{\kappa}_i(F_n):=\kappa_i(F_n)-\kappa_i(Z)$, $i=2,3,4,5,6$. Then
\begin{align}d_{\mathrm{W}}(F_n,Z)&\leq C_1\bigg(\frac{1}{120}\kappa_6(F_n)-\frac{\theta}{6}\kappa_5(F_n)+\frac{1}{3}(2\theta^2-\sigma^2)\kappa_4(F_n)+(2-r)\theta\sigma^2\kappa_3(F_n)\nonumber \\
&\quad+\frac{1}{4}(\kappa_3(F_n))^2-2\theta\kappa_2(F_n)\kappa_3(F_n)+(\sigma^4+4r\theta^2\sigma^2)\kappa_2(F_n)\nonumber \\
\label{doww}&\quad+4\theta^2(\kappa_2(F_n))^2+r^2\theta^2\sigma^4\bigg)^{\frac{1}{2}}+C_2\mathbb{E}|r(\sigma^2+2\theta^2)-\kappa_2(F_n)|,\\
&\leq C_1\bigg(\frac{1}{\sqrt{120}}\sqrt{|\tilde{\kappa}_6(F_n)|}+\frac{\sqrt{|\theta|}}{\sqrt{6}}\sqrt{|\tilde{\kappa}_5(F_n)|}+\frac{\sqrt{|2\theta^2-\sigma^2|}}{\sqrt{3}}\sqrt{|\tilde{\kappa}_4(F_n)|}\nonumber\\
&\quad+\sigma\sqrt{|(2-r)\theta|}\sqrt{|\tilde{\kappa}_3(F_n)|}+\frac{1}{2}|\tilde{\kappa}_3(F_n)|
+\sigma\sqrt{\sigma^2+4r\theta^2}\sqrt{|\tilde{\kappa}_2(F_n)|} \nonumber\\
\label{doww2}&\quad+2|\theta||\tilde{\kappa}_2(F_n)|+\sqrt{2|\theta|}\sqrt{|\kappa_2(F_n)\kappa_3(F_n)-\kappa_2(Z)\kappa_3(Z)|}\bigg)+C_2|\tilde{\kappa}_2(F_n)|,
\end{align}
where $C_1$ and $C_2$ are defined as in Theorem \ref{etthm}. 

Bounds on $d_{\mathrm{K}}(F_n,Z)$ follow immediately from combining inequality (\ref{doww}) or inequality (\ref{doww2}) and the bounds of Proposition \ref{prop1} (with different bounds being used according to the value of $r$).
\end{corollary}

\begin{proof}It is well-known that $\mathbb{E}[\Gamma_2(F_n)]=\kappa_2(F_n)$ (see \cite{np10}), and the equality 
\begin{align*}&\mathbb{E}[(\sigma^2 (F_n+r\theta)+2\theta\Gamma_2(F_n)-\Gamma_3(F_n))^2] \\
&\quad=\frac{1}{120}\kappa_6(F_n)-\frac{\theta}{6}\kappa_5(F_n)+\frac{1}{3}(2\theta^2-\sigma^2)\kappa_4(F_n)+(2-r)\theta\sigma^2\kappa_3(F_n)\\
&\quad\quad+\frac{1}{4}(\kappa_3(F_n))^2-2\theta\kappa_2(F_n)\kappa_3(F_n)+(\sigma^4+4r\theta^2\sigma^2)\kappa_2(F_n)+4\theta^2(\kappa_2(F_n))^2+r^2\theta^2\sigma^4\\
&\quad=:G_{r,\theta,\sigma}(F_n)
\end{align*}
was shown in the proof of Theorem 5.8 of \cite{eichelsbacher}. Substituting these formulas into (\ref{etqu}) gives us (\ref{doww}). 

A simple calculation using the $\mathrm{VG}_c(r,\theta,\sigma)$ cumulant formulas given at the end of Section \ref{sec2} gives that
\begin{align*}G_{r,\theta,\sigma}(F_n)&=\frac{1}{120}\tilde{\kappa}_6(F_n)-\frac{\theta}{6}\tilde{\kappa}_5(F_n)+\frac{1}{3}(2\theta^2-\sigma^2)\tilde{\kappa}_4(F_n)+(2-r)\theta\sigma^2\tilde{\kappa}_3(F_n)\\
&\quad-2\theta\big(\kappa_2(F_n)\kappa_3(F_n)-\kappa_2(Z)\kappa_3(Z)\big)+(\sigma^4+4r\theta^2\sigma^2)\tilde{\kappa}_2(F_n)\\
&\quad+4\theta^2(\tilde{\kappa}_2(F_n))^2.
\end{align*}
Plugging this formula into the upper bound (\ref{doww}), using that $\kappa_2(Z)=r(\sigma^2+2\theta^2)$, and then using the triangle inequality gives us (\ref{doww2}).
\end{proof}

\begin{remark}\label{rtyu}We expect that, for any $r>0$, our bound on $d_{\mathrm{K}}(F_n,Z)$ will be of sub-optimal order.  It is not possible to easily adapt the proof of Theorem \ref{etthm} to obtain Kolmogorov distance bounds with the same rate of convergence as the Wasserstein distance bounds (\ref{etqu}) and (\ref{doww}). This is because the first derivative of the solution $f_z$ of the $\mathrm{VG}_c(r,\theta,\sigma)$ Stein equation with test function $h_z(x)=\mathbf{1}(x\leq z)$ has a discontinuity (see Proposition \ref{disclem}).  This is in contrast to the case of normal approximation, for which bounds on the first derivative of the solution of the normal Stein equation suffice, and optimal order Kolmogorov distance bounds have been obtained \cite{np15}.   
\end{remark}

\begin{remark}Consider the smooth Wasserstein distance $d_{\mathcal{H}_2}(F,G)$ between the distributions of two random elements $F$ and $G$, defined by
\[d_{\mathcal{H}_2}(F,G):=\sup_{h\in\mathcal{H}_2}|\mathbb{E}h(F)-\mathbb{E}h(G)|,\]
where $\mathcal{H}_2=\{h:\mathbb{R}\rightarrow\mathbb{R}\,|\,\text{$h'$ is Lipschitz, $\|h'\|\leq1$, $\|h''\|\leq1$}\}$ (see \cite{arrasetal,dp18}). Note that $d_{\mathcal{H}_2}(F,G)\leq d_{\mathrm{W}}(F,G)$ for any random elements $F$ and $G$ such that $d_{\mathrm{W}}(F,G)$ is well-defined. Let $F_n$ and $Z$ be defined as in Corollary \ref{cor5.4}. In addition, define 
\[\mathbf{M}(F_n)=\max\{|\tilde{\kappa}_i(F_n)|\,:\,i=2,3,4,5,6\}.\]
Recently, 
%using Theorem \ref{thm1} and Corollary \ref{cortothm1} as part of their proof, 
\cite{aet21} have obtained the following rather beautiful VG approximation with optimal rate of convergence: There exist constants $K_1,K_2>0$ only depending on $r$, $\theta$ and $\sigma$ such that
\begin{equation}\label{optvg}K_1\mathbf{M}(F_n)\leq d_{\mathcal{H}_2}(F_n,Z)\leq K_2\mathbf{M}(F_n).
\end{equation}
The upper bound in (\ref{optvg}) improves the bound of Corollary \ref{cor5.4} by removing the square root factor. This improvement comes at the expense of being given with respect to the weaker $d_{\mathcal{H}_2}$ metric. As part of their proof, \cite{aet21} utilised bounds from Theorem \ref{thm1} and Corollary \ref{cortothm1}. In the light of Proposition \ref{ptpt}, it seems that a quite different approach to the one used by \cite{aet21} would be needed to achieve a bound of the form $d_{\mathrm{W}}(F_n,Z)\leq K\mathbf{M}(F_n)$, assuming such a result holds.
\end{remark}

A number of special and limiting cases of VG distributions are given in Proposition 1.2 of \cite{gaunt vg}, and Corollary \ref{cor5.4} can be specialised to these cases. We note two illustrative examples.

\begin{example}The $\mathrm{VG}_c(r,0,\sigma/\sqrt{r})$ distribution converges to the $N(0,\sigma^2)$ distribution as $r\rightarrow\infty$.  It is readily seen that $\lim_{r\rightarrow\infty}A_{r,0,\sigma/\sqrt{r}}=0$ and $\lim_{r\rightarrow\infty}C_{r,0,\sigma/\sqrt{r}}=6$. In this limit, we have that $C_1=2(1+2\cdot6)=26$. Let $F_n=I_2(f_n)$ and suppose that $\mathbb{E}[F_n^2]=\sigma^2$. Then, with $Z\sim N(0,\sigma^2)$, we obtain from (\ref{doww2}) the bound
\begin{align*}d_{\mathrm{W}}(F_n,Z)\leq26\bigg(\frac{1}{\sqrt{120}}\sqrt{|\tilde{\kappa}_6(F_n)|}+\frac{\sigma}{\sqrt{3}}\sqrt{|\tilde{\kappa}_4(F_n)|}+\frac{1}{2}|\kappa_3(F_n)|\bigg).
\end{align*}
As expected, given its derivation  from a general theorem for VG approximation, this result is weaker than the quantitative Gaussian fourth moment theorem of \cite{np09}. It is worth noting that, for $F_n=I_2(f_n)$ and $Z\sim N(0,1)$, we have that, for $k\geq3$, $|\mathbb{E}[F_n^k]-\mathbb{E}[Z^k]|\leq c_{k}\sqrt{\mathbb{E}[F_n^4]-3}$, where $c_k>0$ is an explicit constant depending only on $k$ (see \cite{np10b}). Therefore, for $k=,3,6$, $\tilde{\kappa}_k(F_n)\leq c_k'\sqrt{\mathbb{E}[F_n^4]-3}$ (for some $c_k'>0$), which is consistent with the famous condition of \cite{np05} that convergence in distribution of a sequence of random variables, with zero mean and unit variance, living in a Wiener chaos of fixed order to the standard Gaussian distribution occurs if and only if the sequence of fourth moments convergences to that of a $N(0,1)$ random variable.
\end{example}

\begin{example}The $\mathrm{VG}_c(2,0,\sigma)$ distribution corresponds to the $\mathrm{Laplace}(0,\sigma)$ distribution with density $p(x)=\frac{1}{2\sigma}\mathrm{e}^{-|x|/\sigma}$, $x\in\mathbb{R}$.  We have that $A_{2,0,\sigma}=\frac{2\sqrt{\pi}}{\sqrt{3}}$ and $C_{2,0,\sigma}=8+\frac{16\sqrt{\pi}}{\sqrt{3}}$. In this case, we have $C_1=\frac{134.978\ldots}{\sigma^2}<\frac{135}{\sigma^2}$. Let $F_n=I_2(f_n)$ be such that $\mathbb{E}[F_n^2]=2\sigma^2$. Then, with $Z\sim\mathrm{Laplace}(0,\sigma)$, we obtain from (\ref{doww2}) the bound
\begin{align*}d_{\mathrm{W}}(F_n,Z)\leq\frac{135}{\sigma^2}\bigg(\frac{1}{\sqrt{120}}\sqrt{|\tilde{\kappa}_6(F_n)|}+\frac{\sigma}{\sqrt{3}}\sqrt{|\tilde{\kappa}_4(F_n)|}+\frac{1}{2}|\kappa_3(F_n)|\bigg).
\end{align*}
\end{example}

%\subsection{The generalized Rosenblatt process at extreme critical exponent}

%In this section we consider an application of our bounds to obtain bounds on the rate of convergence, in the Wasserstein and Kolmogorov distances, in a recent result of \cite{}. In doing so, we complement the recent work \cite{} which also obtained bounds on the rate of convergence in the Wasserstein-2 distance.

We end this section by demonstrating how Corollary \ref{cor5.4} can be used to obtain bounds on the rate of convergence in a recent result of \cite{bt17}.

\begin{example}[The generalized Rosenblatt process at extreme critical
exponent] Consider the Rosenblatt process $Z_{\gamma_1,\gamma_2}(t)$, introduced by \cite{mt12} as the double Wiener-It\^{o} integral 
\begin{equation*}Z_{\gamma_1,\gamma_2}(t)=\int_{\mathbb{R}^2}^{\prime}\bigg(\int_0^t(s-x_1)_+^{\gamma_1}(s-x_2)_+^{\gamma_2}\,\mathrm{d}s\bigg)\,\mathrm{d}B_{x_1}\,\mathrm{d}B_{x_2},
\end{equation*}
where the prime $\prime$ indicates exclusion of the diagonals $x_1=x_2$ in the stochastic integral, $B_{x}$ is standard Brownian motion and $\gamma_i\in(-1,-\frac{1}{2})$, $i=1,2$, and $\gamma_1+\gamma_2>-\frac{3}{2}$. The Rosenblatt process \cite{t75} is the special case $Z_\gamma(t)=Z_{\gamma,\gamma}(t)$, $-\frac{3}{4}<\gamma<-\frac{1}{2}$. By a change of variables and using the scale invariant property of Brownian motion it can be shown that
\begin{equation*}Z_{\gamma_1,\gamma_2}(t)=t^{2+\gamma_1+\gamma_2}\int_{\mathbb{R}^2}^{\prime}\bigg(\int_0^1(s-x_1)_+^{\gamma_1}(s-x_2)_+^{\gamma_2}\,\mathrm{d}s\bigg)\,\mathrm{d}B_{x_1}\,\mathrm{d}B_{x_2},
\end{equation*}
and so $Z_{\gamma_1,\gamma_2}(t)=_d t^{2+\gamma_1+\gamma_2}Z_{\gamma_1,\gamma_2}(1)$. From now on, for simplicity, we will work with the random variable $Z_{\gamma_1,\gamma_2}(1)$; results for the general $t>0$ case can be inferred from a rescaling.
%; extending to the general $t>0$ case is straightforward. 
For $\rho\in(0,1)$, define the random variable $Y_\rho$ by
\begin{equation*}Y_\rho=\frac{a_\rho}{\sqrt{2}}(X_1-1)-\frac{b_\rho}{\sqrt{2}}(X_2-1),
\end{equation*}
where $X_1$ and $X_2$ are independent $\chi_{(1)}^2$ random variables and 
\begin{align*}a_\rho=\frac{(2\sqrt{\rho})^{-1}+(\rho+1)^{-1}}{\sqrt{(2\rho)^{-1}+2(\rho+1)^{-2}}}, \quad b_\rho=\frac{(2\sqrt{\rho})^{-1}-(\rho+1)^{-1}}{\sqrt{(2\rho)^{-1}+2(\rho+1)^{-2}}}.
\end{align*}
We follow \cite{aaps17} and suppose for simplicity that $\gamma_1\geq\gamma_2$ and that $\gamma_2=(\gamma_1+\frac{1}{2})/\rho-\frac{1}{2}$.

It was recently shown by \cite{aaps17} that, as $\gamma_1\rightarrow-\frac{1}{2}$,
\begin{equation}\label{dw2}d_{\mathrm{W}_2}(Z_{\gamma_1,\gamma_2}(1),Y_\rho)\leq C_\rho\sqrt{-\gamma_1-\frac{1}{2}},
\end{equation}
where $C_\rho>0$ is a constant depending solely on $\rho$ and $d_{\mathrm{W}_2}$ is the Wasserstein-2 distance.
% It is worth noting that \cite{aaps17} work in the case $t=1$, but that results for the general $t>0$ case can be inferred by a simple rescaling. 
(Note that if $\gamma_1\rightarrow-\frac{1}{2}$, then automatically $\gamma_2\rightarrow-\frac{1}{2}$.) Working with respect to the weaker $d_{\mathcal{H}_2}$ metric, \cite{aet21} have very recently obtained a faster rate of convergence: as $\gamma_1\rightarrow-\frac{1}{2}$,
\[d_{\mathcal{H}_2}(Z_{\gamma_1,\gamma_2}(1),Y_\rho)\leq C_\rho\bigg|-\gamma_1-\frac{1}{2}\bigg|\]
 With these results, \cite{aaps17} and \cite{aet21} have given bounds on the rate of convergence in a recent limit theorem of \cite[Theorem 2.4]{bt17}. To obtain the bound (\ref{dw2}), \cite{aaps17} showed that, for any $m\geq2$, as $\gamma_1\rightarrow-\frac{1}{2}$,
\begin{equation}\label{sfbfbf}\kappa_m(Z_{\gamma_1,\gamma_2}(1))=\kappa_m(Y_\rho)+O\Big(-\gamma-\frac{1}{2}\Big),
\end{equation}
and inserted this asymptotic relation into a general Wasserstein-2 distance bound (expressed in terms of cumulants) in which the limit distribution can be represented as linear combinations of centered chi-square random variables. (The statement of the asymptotic relation (\ref{sfbfbf}) in \cite{aaps17} is only given for $m\geq3$, but on inspecting their proof it can be seen that the asymptotic relation is also valid in the case $m=2$.) Similarly, \cite{aet21}, substituted (\ref{sfbfbf}) into the upper bound of (\ref{optvg}).  As the Wasserstein-2 metric is stronger than the Wasserstein metric, it is immediate that, as $\gamma_1\rightarrow-\frac{1}{2}$,
\begin{equation}\label{dw22}d_{\mathrm{W}}(Z_{\gamma_1,\gamma_2}(1),Y_\rho)\leq C_\rho\sqrt{-\gamma_1-\frac{1}{2}}.
\end{equation}

We now show that we can apply bound (\ref{doww2}) of Corollary \ref{cor5.4} to obtain an alternative proof  of (\ref{dw22}). This is a weaker result than that of \cite{aaps17}, but the example is useful in demonstrating the applicability of Corollary \ref{cor5.4}. We also apply Proposition \ref{prop1} to obtain a bound on the rate of convergence in the Kolmogorov distance, which is a new result.

We first recognise $Y_\rho$ as a VG random variable. Let $\Gamma(r,\lambda)$ denote a gamma random variable with density $p(x)=\frac{1}{\Gamma(r)}\lambda^rx^{r-1}\mathrm{e}^{-\lambda x}$, $x>0$. Then, Proposition 1.2 of \cite{gaunt vg} tells us that if $G_1\sim\Gamma(r,\lambda_1)$ and $G_2\sim\Gamma(r,\lambda_2)$ are independent, then $G_1-G_2\sim\mathrm{VG}(r,(2\lambda_1)^{-1}-(2\lambda_2)^{-1},(\lambda_1\lambda_2)^{-1/2},0)$. Observe that $a_\rho>0$, $b_\rho>0$ and that, since $\chi_{(1)}^2=_d\Gamma(\frac{1}{2},\frac{1}{2})$, we have that $\frac{a_\rho}{\sqrt{2}}X_1\sim \Gamma(\frac{1}{2},\frac{1}{\sqrt{2}a_\rho})$ and $\frac{b_\rho}{\sqrt{2}}X_2\sim \Gamma(\frac{1}{2},\frac{1}{\sqrt{2}b_\rho})$.  Therefore, $Y_\rho\sim \mathrm{VG}_c(1,\frac{a_\rho-b_\rho}{\sqrt{2}},\sqrt{2a_\rho b_\rho})$. As $Z_{\gamma_1,\gamma_2}(1)$ is a double Wiener-It\^{o} integral (and hence also satisfies $\mathbb{E}[Z_{\gamma_1,\gamma_2}(1)]=0$), we may apply bound (\ref{doww2}) of Corollary \ref{cor5.4} together with the asymptotic relation (\ref{sfbfbf}) to obtain (\ref{dw22}).  By part (ii) of Proposition \ref{prop1} (note that here $r=1$) we can then obtain that,
as $\gamma_1\rightarrow-\frac{1}{2}$,
\begin{equation*}d_{\mathrm{K}}(Z_{\gamma_1,\gamma_2}(1),Y_\rho)\leq C_\rho'\bigg(-\gamma_1-\frac{1}{2}\bigg)^{\frac{1}{4}}\log\bigg(\frac{1}{-\gamma_1-\frac{1}{2}}\bigg),
\end{equation*}
where $C_\rho'>0$ depends only on $\rho$.
%We expect that this rate to be of sub-optimal order.
%, but it is the first bound in the literature on the rate of convergence in the Kolmogorov metric for the generalized Rosenblatt process at extreme critical exponent $\gamma\downarrow-\frac{1}{2}$.
% of $Z_{\gamma_1,\gamma_2}(1)$ to $Y_\rho$ in the Kolmogorov metric in the literature.
\end{example}

%We can, however, extract a quantitative six moment theorem in the Kolmogorov distance by using the bounds of Proposition \ref{prop1} to the bound (\ref{doww}) (with different bounds being used according to the value of $r$).  It is unlikely that these bounds will be of the optimal order. 

\section{Further proofs}\label{sec6}

\noindent\emph{Proof of Proposition \ref{prop2mode}.} We first prove inequality (\ref{vgpdfineq2}). We prove the result for the case $\theta\not=0$ and then treat the case $\theta=0$. Let $r>2$ and $\sigma>0$. We prove the result for $\mu=0$; the extension to general $\mu\in\mathbb{R}$ is obvious. We also fix $\theta>0$; the case $\theta<0$ is very similar because, for $Z\sim\mathrm{VG}(r,\theta,\sigma,0)$, we have that $-Z\sim\mathrm{VG}(r,-\theta,\sigma,0)$.  Recall that, for $x>0$, the $\mathrm{VG}(r,\theta,\sigma,0)$ density is given by
\begin{equation*}p(x) = \frac{1}{\sigma\sqrt{\pi} \Gamma(\frac{r}{2})}\bigg(\frac{\sigma^2}{2(\theta^2+\sigma^2)}\bigg)^{\frac{r-1}{2}} u(x)v(x),
\end{equation*}
where
\begin{align*}u(x)=\mathrm{e}^{\frac{\theta}{\sigma^2} x}, \quad v(x)=\bigg(\frac{\sqrt{\theta^2 + \sigma^2}}{\sigma^2} x\bigg)^{\frac{r-1}{2}} K_{\frac{r-1}{2}}\bigg(\frac{\sqrt{\theta^2 + \sigma^2}}{\sigma^2}x \bigg).
\end{align*}
It suffices to consider $x>0$, because the mode of the $\mathrm{VG}(r,\theta,\sigma,0)$ distribution is strictly positive for $\theta>0$.  By the upper bound in (\ref{gmineq}), we know that $M<\theta(r-2)$. Now, $u(x)$ is a strictly increasing function of $x$ on $(0,\infty)$, and $v(x)$ is a strictly decreasing function of $x$ on $(0,\infty)$ (see (\ref{ddbk})). 
Therefore, for all $x>0$,
%For $\theta\not=0$, inequality (\ref{vgpdfineq2}) is now obtained by observing that, for $x>0$, 
\begin{align*}p(x) < \frac{1}{\sigma\sqrt{\pi} \Gamma(\frac{r}{2})}\bigg(\frac{\sigma^2}{2(\theta^2+\sigma^2)}\bigg)^{\frac{r-1}{2}} u(\theta(r-2))\lim_{x\downarrow0}v(x),
\end{align*}
where $\lim_{x\downarrow0}v(x)$ can be calculated using the limiting form (\ref{Ktend0}). As it sufficed to consider $x>0$ and $\theta>0$, the proof of inequality (\ref{vgpdfineq2}) is complete for the case $\theta\not=0$. To extend the range of validity of inequality (\ref{vgpdfineq2}) to $\theta\in\mathbb{R}$, we use that in the $\theta=0$ case the mode of the $\mathrm{VG}(r,\theta,\sigma,0)$ distribution is 0.  Letting $x\rightarrow0$ in the VG density using (\ref{Ktend0}) gives that, for $r>2$, $\theta=0$, $\sigma>0$,
\[p(x)\leq\lim_{x\rightarrow0}p(x)=\frac{\Gamma(\frac{r-1}{2})}{2\sigma\sqrt{\pi} \Gamma(\frac{r}{2})},\]
which verifies that inequality (\ref{vgpdfineq2}) is also valid for $\theta=0$.  

We now prove inequality (\ref{vgpdfineq}). Suppose that $r>3$ and $\theta\not=0$.  As in the proof of inequality (\ref{vgpdfineq2}) it will suffice the treat the case $\theta>0$ and $x>0$. By the two-sided inequality (\ref{gmineq}), we know that $\theta(r-3)<M<\theta(r-2)$. Recall that $u(x)$ is a strictly increasing function of $x$ on $(0,\infty)$, and $v(x)$ is a strictly decreasing function of $x$ on $(0,\infty)$. Therefore, for $x>0$,
\begin{align*}p(x) < \frac{1}{\sigma\sqrt{\pi} \Gamma(\frac{r}{2})}\bigg(\frac{\sigma^2}{2(\theta^2+\sigma^2)}\bigg)^{\frac{r-1}{2}} u(\theta(r-2))v(\theta(r-3)).
\end{align*}
On evaluating $u(\theta(r-2))v(\theta(r-3))$, we obtain the upper bound in (\ref{vgpdfineq}). As it sufficed to consider $x>0$ and $\theta>0$, this completes the proof of inequality (\ref{vgpdfineq}).  

 Finally, the assertion that inequality (\ref{vgpdfineq2}) is less accurate than (\ref{vgpdfineq}) for $r>3$, $\theta\not=0$ follows because $v(x)$ is a strictly decreasing function of $x$ on $(0,\infty)$.  \hfill $\Box$

\vspace{3mm} 

\noindent\emph{Proof of Proposition \ref{disclem}.} To simplify the notation we set $\mu=0$; the general case follows from a simple translation. To further simplify the notation, we shall work with the change of parameters (\ref{parameter}) and set $\alpha=1$, so that $|\beta|<1$, with the general $\alpha>0$ case following from rescaling. With this change of parameters, the solution of the $\mathrm{VG}(r,\theta,\sigma,0)$ Stein equation with test function $h_z(x)=\mathbf{1}(x\leq z)$ is given by
\begin{align}f_z(x) &=-\frac{\mathrm{e}^{-\beta x} K_{\nu}(|x|)}{\sigma^2|x|^{\nu}} \int_0^x \mathrm{e}^{\beta t} |t|^{\nu} I_{\nu}(|t|) [\mathbf{1}(t\leq z)-\mathbb{P}(Z\leq z)] \,\mathrm{d}t \nonumber \\
\label{ksoln}&\quad-\frac{\mathrm{e}^{-\beta x} I_{\nu}(|x|)}{\sigma^2|x|^{\nu}} \int_x^{\infty} \mathrm{e}^{\beta t} |t|^{\nu} K_{\nu}(|t|)[\mathbf{1}(t\leq z)-\mathbb{P}(Z\leq z)]\,\mathrm{d}t.
\end{align}
We now set $z=0$. Differentiating (\ref{ksoln}) using the formulas (\ref{diff11}) and (\ref{diff22}) gives us
\begin{align*}f_0'(x)&=\frac{\mathrm{e}^{-\beta x}}{\sigma^2}\bigg(\beta\frac{K_{\nu}(|x|)}{|x|^\nu}+\frac{K_{\nu+1}(|x|)}{|x|^\nu}\mathrm{sgn}(x)\bigg)\int_0^x\mathrm{e}^{\beta t}|t|^\nu I_\nu(|t|)[\mathbf{1}(t\leq 0)-\mathbb{P}(Z\leq 0)]\,\mathrm{d}t \\
&\quad+\frac{\mathrm{e}^{-\beta x}}{\sigma^2}\bigg(\beta\frac{I_{\nu}(|x|)}{|x|^\nu}-\frac{I_{\nu+1}(|x|)}{|x|^\nu}\mathrm{sgn}(x)\bigg)\int_x^\infty \mathrm{e}^{\beta t}|t|^\nu K_\nu(|t|)[\mathbf{1}(t\leq 0)-\mathbb{P}(Z\leq 0)]\,\mathrm{d}t.
\end{align*}

We have that, for all $\nu>-\frac{1}{2}$ and $-1<\beta<1$,
\begin{align*}\lim_{x\rightarrow0}\bigg[\frac{\mathrm{e}^{-\beta x}K_{\nu}(|x|)}{|x|^\nu}\int_0^x \mathrm{e}^{\beta t}|t|^\nu I_\nu(|t|)[\mathbf{1}(t\leq 0)-\mathbb{P}(Z\leq 0)]\,\mathrm{d}t\bigg]&=0,\\
\lim_{x\rightarrow0}\bigg[\frac{\mathrm{e}^{-\beta x}I_{\nu+1}(|x|)}{|x|^\nu}\int_x^\infty \mathrm{e}^{\beta t}|t|^\nu K_\nu(|t|)[\mathbf{1}(t\leq 0)-\mathbb{P}(Z\leq 0)]\,\mathrm{d}t\bigg]&=0.
\end{align*}
Here the first limit is readily seen to be equal to 0 through an application of the limiting forms (\ref{Itend0}) and (\ref{Ktend0}), whilst the second limit can be seen to be equal to 0 through an application of (\ref{Itend0}) and by identifying $\mathrm{e}^{\beta}|t|^\nu K_\nu(|t|)$ as a constant multiple of the $\mathrm{VG}(r,\theta,\sigma,0)$ density, which means that the integral must be bounded for all $x\in\mathbb{R}$. The term
\begin{equation*}J(x):=\frac{\beta\mathrm{e}^{-\beta x}I_{\nu}(|x|)}{\sigma^2|x|^\nu}\int_x^\infty \mathrm{e}^{\beta t}|t|^\nu K_\nu(|t|)[\mathbf{1}(t\leq 0)-\mathbb{P}(Z\leq 0)]\,\mathrm{d}t
\end{equation*}
is the product of two functions that are continuous at $x=0$, 
\[u(x)=\frac{\beta\mathrm{e}^{-\beta x}I_{\nu}(|x|)}{\sigma^2|x|^\nu}, \quad v(x)=\int_x^\infty \mathrm{e}^{\beta t}|t|^\nu K_\nu(|t|)[\mathbf{1}(t\leq 0)-\mathbb{P}(Z\leq 0)]\,\mathrm{d}t,\]
meaning that $J(0-)=J(0+)$. Therefore
\begin{eqnarray*}f_0'(0+)&=&-\mathbb{P}(Z\leq0)\lim_{x\downarrow0}\bigg[\frac{\mathrm{e}^{-\beta x}K_{\nu+1}(x)}{\sigma^2x^\nu}\int_0^x \mathrm{e}^{\beta t}t^\nu I_\nu(t)\,\mathrm{d}t\bigg]+J(0+), \\
f_0'(0-)&=&-\big(1-\mathbb{P}(Z\leq0)\big)\lim_{x\uparrow0}\bigg[\frac{\mathrm{e}^{-\beta x}K_{\nu+1}(-x)}{\sigma^2(-x)^\nu}\int_0^x\mathrm{e}^{\beta t}(-t)^\nu I_\nu(-t)\,\mathrm{d}t\bigg]+J(0-) \\
&=&\big(1-\mathbb{P}(Z\leq0)\big)\lim_{x\uparrow0}\bigg[\frac{\mathrm{e}^{-\beta x}K_{\nu+1}(-x)}{\sigma^2(-x)^\nu}\int_0^{-x}\mathrm{e}^{-\beta u}u^\nu I_\nu(u)\,\mathrm{d}u\bigg]+J(0+).
\end{eqnarray*}
The above limits can be caluclated using (\ref{Itend0}) and (\ref{Ktend0}), which gives $f_0'(0+)=-\frac{1}{\sigma^2(2\nu+1)}\mathbb{P}(Z\leq0)+J(0+)$ and $f_0'(0-)=\frac{1}{\sigma^2(2\nu+1)}(1-\mathbb{P}(Z\leq0))+J(0+)$,  thus proving the assertion. \hfill $\Box$ 

\vspace{3mm} 

\noindent\emph{Proof of Proposition \ref{ptpt}.} Again, we set $\mu=0$. To simplify the expressions, we shall also work with a rescaling of the solution $g(x):=\sigma^2f(x)$, which will remove a multiplicative constant of $\frac{1}{\sigma^2}$ from the calculations. The analogous approach to the proof of Proposition \ref{disclem} would be to find a Lipschitz test function $h$ for which $g''$ has a discontinuity. This would be quite a tedious undertaking, and instead we choose a highly oscillating test function and perform an asymptotic analysis. Let $h(x)=\frac{\sin(ax)}{a}\in\mathcal{H}_{\mathrm{W}}$.  If a general bound of the form $\|g^{(3)}\|\leq M_{r,\theta,\sigma}\|h'\|$ was available, then we would be able to find a constant $N_{r,\theta,\sigma}>0$, that does not involve $a$, such that $\|g^{(3)}\|\leq N_{r,\theta,\sigma}$. We will show that such a bound is not possible by showing that, with the choice of test function $h(x)=\frac{\sin(ax)}{a}$, the third derivative $g^{(3)}(x)$ blows up if we let $a\gg1$ and choose $x$ such that $ax\ll 1\ll a^2x$.
%We will show that $g^{(3)}(x)$ blows up in the limit $x\rightarrow0$ for $a$ such that $ax\ll 1\ll a^2x$. 
 This means that it is not possible to obtain such a bound for $\|f^{(3)}\|$, which will prove the proposition.  Before beginning this analysis, it is worth noting that $h''(x)=-a\sin(ax)$ blows up if $a\gg1$ and $x$ is chosen such that $ax\ll 1\ll a^2 x$, which can be seen from the expansion $\sin(t)=t+O(t^3)$, $t\rightarrow0$.  It is therefore still possible that a general bound of the form $\|g^{(3)}\|\leq M_{r,\theta,\sigma,0}\|\tilde{h}\|+M_{r,\theta,\sigma,1}\|h'\|+M_{r,\theta,\sigma,2}\|h''\|$ can be obtained. Indeed, such a bound has been obtained; see Section 3.1.7 of \cite{dgv15}.  

%A direct differentation of the solution (\ref{}) and some simplifications using elementary properties of modified Bessel functions yields (see Lemma 3.16 of \cite{gaunt thesis}):
 Let $x>0$.  We first obtain a formula for $g^{(3)}(x)$. We have already obtained a formula for $g'(x)$  (see (\ref{fdash})), and differentiating this formula and then simplifying using the differentiation formulas (\ref{diff11}) and (\ref{diff22}) followed by an application of the Wronskian formula $I_\nu(x)K_{\nu+1}(x)+I_{\nu+1}(x)K_\nu(x)=\frac{1}{x}$ \cite{olver} gives
\begin{align*}g''(x) &= \frac{\tilde{h}(x)}{x} -\bigg[\frac{\mathrm{d}^2}{\mathrm{d}x^2}\bigg(\frac{\mathrm{e}^{-\beta x} K_{\nu}(x)}{x^{\nu}}\bigg)\bigg] \int_0^x\mathrm{e}^{\beta t} t^{\nu} I_{\nu}(t)\tilde{h}(t)\,\mathrm{d}t\\
&\quad -\bigg[\frac{\mathrm{d}^2}{\mathrm{d}x^2}\bigg(\frac{\mathrm{e}^{-\beta x}I_{\nu }(x)}{x^{\nu}}\bigg)\bigg] \int_x^{\infty}\mathrm{e}^{\beta t} t^{\nu} K_{\nu}(t)\tilde{h}(t)\,\mathrm{d}t.
\end{align*}
Differentiating again gives
\begin{align}g^{(3)}(x)&=\frac{h'(x)}{x}-\frac{\tilde{h}(x)}{x^2} -\bigg[\frac{\mathrm{d}^3}{\mathrm{d}x^3}\bigg(\frac{ \mathrm{e}^{-\beta x}K_{\nu}(x)}{x^{\nu}}\bigg)\bigg] \int_0^x \mathrm{e}^{\beta t}t^{\nu} I_{\nu}(t)\tilde{h}(t)\,\mathrm{d}t+R_1\nonumber \\
&\quad+\tilde{h}(x)\bigg\{-x^{\nu}I_{\nu}(x)\frac{\mathrm{d}^2}{\mathrm{d}x^2}\bigg(\frac{ \mathrm{e}^{-\beta x}K_{\nu}(x)}{x^{\nu}}\bigg)+x^{\nu}K_{\nu}(x)\frac{\mathrm{d}^2}{\mathrm{d}x^2}\bigg(\frac{\mathrm{e}^{-\beta x} I_{\nu}(x)}{x^{\nu}}\bigg)\bigg\} \nonumber\\
&=\frac{h'(x)}{x}-\bigg(\frac{2\nu+2}{x^2}+\frac{2\beta}{x}\bigg)\tilde{h}(x) -\bigg[\frac{\mathrm{d}^3}{\mathrm{d}x^3}\bigg(\frac{\mathrm{e}^{-\beta x}K_{\nu}(x)}{x^{\nu}}\bigg)\bigg] \int_0^x \mathrm{e}^{\beta t}t^{\nu} I_{\nu}(t)\tilde{h}(t)\,\mathrm{d}t\nonumber\\
\label{ff33} &\quad+R_1,
\end{align} 
where
\begin{align*}R_1= -\bigg[\frac{\mathrm{d}^3}{\mathrm{d}x^3}\bigg(\frac{\mathrm{e}^{-\beta x}I_{\nu }(x)}{x^{\nu}}\bigg)\bigg] \int_x^{\infty} \mathrm{e}^{\beta t}t^{\nu} K_{\nu}(t)\tilde{h}(t)\,\mathrm{d}t.
\end{align*}
Here, in simplifying to obtain the formula (\ref{ff33}) we used the differentiation formulas (\ref{diffi2e}) and (\ref{diffk2e}) followed by an application of the Wronskian formula. We can bound $R_1$ using inequalities (\ref{difiineq}) and (\ref{rnmt1}) to obtain that, for all $\nu>-\frac{1}{2}$, $-1<\beta<1$ and $x>0$, 
\begin{align*}|R_1|&\leq \|\tilde{h}\|\bigg[\frac{\mathrm{d}^3}{\mathrm{d}x^3}\bigg(\frac{\mathrm{e}^{-\beta x}I_{\nu }(x)}{x^{\nu}}\bigg)\bigg] \int_x^{\infty} \mathrm{e}^{\beta t}t^{\nu} K_{\nu}(t)\,\mathrm{d}t\\
&\leq 8\|\tilde{h}\|\frac{\mathrm{e}^{-\beta x}I_{\nu }(x)}{x^{\nu}} \int_x^{\infty} \mathrm{e}^{\beta t}t^{\nu} K_{\nu}(t)\,\mathrm{d}t\leq 8M_{\nu,\beta}\|\tilde{h}\|,
\end{align*}
where $M_{\nu,\beta}$ is defined in (\ref{mdefn}). We have that $\|\tilde{h}\|\leq 2\|h\|=\frac{2}{a}$, and so the term $R_1$ does not blow up in the limit $a\rightarrow\infty$.

% For all $\nu>-\frac{1}{2}$, the expression involving modified Bessel functions in (\ref{r1one}) is uniformly bounded for all $x\geq0$ (see Lemma D.4 of \cite{gaunt thesis}), and $\|\tilde{h}\|\leq 2\|h\|=\frac{2}{a}$.  Hence, the term $R_1$ does not explode when $a\rightarrow\infty$.

%\begin{align}|R_1|&=\bigg|\bigg[\frac{\mathrm{d}^3}{\mathrm{d}x^3}\bigg(\frac{I_{\nu }(x)}{x^{\nu}}\bigg)\bigg] \int_x^{\infty} t^{\nu} K_{\nu}(t)\tilde{h}(t)\,\mathrm{d}t\bigg| \nonumber\\
%\label{r1one}&\leq \|\tilde{h}\|\bigg[\frac{\mathrm{d}^3}{\mathrm{d}x^3}\bigg(\frac{I_{\nu }(x)}{x^{\nu}}\bigg)\bigg] \int_x^{\infty} t^{\nu} K_{\nu}(t)\,\mathrm{d}t.
%\end{align}
%Here, to obtain equality (\ref{ff33}) we used differentiation formulas (\ref{diff11}) and (\ref{diff22}) followed again by the Wronskian formula. For all $\nu>-\frac{1}{2}$, the expression involving modified Bessel functions in (\ref{r1one}) is uniformly bounded for all $x\geq0$ (see Lemma D.4 of \cite{gaunt thesis}), and $\|\tilde{h}\|\leq 2\|h\|=\frac{2}{a}$.  Hence, the term $R_1$ does not explode when $a\rightarrow\infty$.

An application of integration by parts to (\ref{ff33}) gives that
\begin{equation*}g^{(3)}(x)=\frac{h'(x)}{x}+\bigg[\frac{\mathrm{d}^3}{\mathrm{d}x^3}\bigg(\frac{\mathrm{e}^{-\beta x}K_{\nu}(x)}{x^{\nu}}\bigg)\bigg]\int_0^xh'(u)\int_0^u \mathrm{e}^{\beta t}t^\nu I_\nu(t)\,\mathrm{d}t\,\mathrm{d}u+R_1+R_2,
\end{equation*}
where
\begin{align*}R_2&=-\tilde{h}(x)\bigg\{\frac{2\nu+2}{x^2}+\frac{2\beta}{x}+ \bigg[\frac{\mathrm{d}^3}{\mathrm{d}x^3}\bigg(\frac{\mathrm{e}^{-\beta x}K_{\nu}(x)}{x^{\nu}}\bigg)\bigg]\int_0^x \mathrm{e}^{\beta t}t^{\nu} I_{\nu}(t)\,\mathrm{d}t\bigg\}\nonumber\\
&=:-\tilde{h}(x)T_{\nu,\beta}(x).
\end{align*}
We will show that, for all $\nu>-\frac{1}{2}$ and $-1<\beta<1$,  there exists a constant $C_{\nu,\beta}>0$, that does not involve $x$, such that $T_{\nu,\beta}(x)\leq C_{\nu,\beta}$ for all $x>0$.  
For this purpose, it will be sufficient to examine the function $T_{\nu,\beta}(x)$ in the limits $x\downarrow0$ and $x\rightarrow\infty$. 
We have that $T_{\nu,\beta}(x)\rightarrow0$ as $x\rightarrow\infty$. This can be shown by using the differentiation formula (\ref{dk3e}) followed by an application of the limiting form (\ref{Ktendinfinity}) and the following limiting form (see \cite{gaunt pams}). For $\nu>-\frac{1}{2}$, $-1<\beta<1$, we have that, as $x\rightarrow\infty$,
\begin{equation}\label{pamsi}\int_0^x \mathrm{e}^{\beta t}t^{\nu} I_{\nu}(t)\,\mathrm{d}t\sim \frac{1}{\sqrt{2\pi}(1+\beta)}x^{\nu-\frac{1}{2}}\mathrm{e}^{(1+\beta)x}.
\end{equation}
In addition, by applying the differentiation formula (\ref{dk3e}) and then the limiting forms (\ref{Itend0}) and (\ref{Ktend0}) together with the expansion $\mathrm{e}^{-\beta x}=1-\beta x+O(x^2)$, as $x\rightarrow0$, we obtain that, for $\nu>-\frac{1}{2}$, $-1<\beta<1$, as $x\downarrow0$,
\begin{align*}&\bigg[\frac{\mathrm{d}^3}{\mathrm{d}x^3}\bigg(\frac{\mathrm{e}^{-\beta x}K_{\nu}(x)}{x^{\nu}}\bigg)\bigg]\int_0^x \mathrm{e}^{\beta t}t^{\nu} I_{\nu}(t)\,\mathrm{d}t \\
&\quad=-\mathrm{e}^{-\beta x}\bigg\{\bigg(3\beta+\beta^3+\frac{2\nu+1}{x}\bigg)\frac{K_{\nu}(x)}{x^\nu}\nonumber\\
&\quad\quad+\bigg(1+3\beta^2+\frac{3\beta(2\nu+1)}{x}+\frac{(2\nu+1)(2\nu+2)}{x^2}\bigg)\frac{K_{\nu+1}(x)}{x^\nu}\bigg\}\int_0^x \mathrm{e}^{\beta t}t^{\nu} I_{\nu}(t)\,\mathrm{d}t \\
&\quad=-\bigg\{(1-\beta x)\bigg(\frac{(2\nu+1)(2\nu+2)}{x^{2}}+\frac{3\beta(2\nu+1)}{x}\bigg) \frac{2^\nu\Gamma(\nu+1)}{x^{2\nu+1}}+O(x^{-2\nu-1})\bigg\}\times\\
&\quad\quad\times\int_0^x \bigg(\frac{t^{2\nu}}{2^\nu\Gamma(\nu+1)}+\frac{\beta t^{2\nu+1}}{2^\nu\Gamma(\nu+1)}+O(t^{2\nu+2})\bigg)\,\mathrm{d}t \\
&\quad=-\bigg(\frac{(2\nu+1)(2\nu+2)}{x^{2\nu+3}}-\frac{\beta(2\nu+1)(2\nu-1)}{x^{2\nu+2}}\bigg)\bigg(\frac{x^{2\nu+1}}{2\nu+1}+\frac{\beta x^{2\nu+2}}{2\nu+2}\bigg)+O(1)\\
&\quad=-\frac{2\nu+2}{x^2}-\frac{2\beta}{x}+O(1).
\end{align*}
One needs to argue carefully that the remainder term in the curly brackets in the second equality is $O(x^{-2\nu-1})$.  For $\nu>0$, this is justified because we have the expansions $K_{\nu}(x)=2^{\nu-1}\Gamma(\nu)x^{-\nu}+O(x^{-\nu+2})$ and $K_{\nu+1}(x)=2^\nu\Gamma(\nu+1)x^{-\nu-1}+O(x^{-\nu+1})$, as $x\downarrow0$ (see (\ref{Ktend0})). However, for $-\frac{1}{2}<\nu\leq0$ the second term in the asymptotic expansion of $K_{\nu+1}(x)$ is larger than $O(x^{-\nu+1})$, as $x\downarrow0$, so we need to work a little harder. For $-\frac{1}{2}<\nu\leq0$, we use the identity (\ref{Kiden}) followed by the limiting form (\ref{Ktend0}) to get that, as $x\downarrow0$,
\begin{align*}(2\nu+1)\frac{K_{\nu}(x)}{x^{\nu+1}}+(2\nu+1)(2\nu+2)\frac{K_{\nu+1}(x)}{x^{\nu+2}}&=(2\nu+1)\frac{K_{\nu+2}(x)}{x^{\nu+1}}\\
&=\frac{(2\nu+1)2^{\nu+1}\Gamma(\nu+2)}{x^{2\nu+3}}+O(x^{-2\nu-1})\\
&=\frac{(2\nu+1)(2\nu+2)2^\nu\Gamma(\nu+1)}{x^{2\nu+3}}+O(x^{-2\nu-1}),
\end{align*}
as required. This argument is of course also valid for $\nu>0$. Thus, we have show that $T_{\nu,\beta}(x)$ is bounded as $x\downarrow0$, as well as in the limit $x\rightarrow\infty$, and so we have been able to shown that $R_2$ does not blow up when $a\rightarrow\infty$.

%For all $\nu>-\frac{1}{2}$, the expression in the braces in (\ref{r2two}) is uniformly bounded for all $x\geq0$ (see \cite{gaunt thesis}, Lemma D.16).  Therefore, the term $R_2$ does not explode when $a\rightarrow\infty$.

From the differentiation formula (\ref{dk3e}), we have
\begin{align*}g^{(3)}(x)&=\frac{h'(x)}{x}-\frac{(2\nu+1)(2\nu+2)\mathrm{e}^{-\beta x}K_{\nu+1}(x)}{x^{\nu+2}}\int_0^xh'(u)\int_0^u \mathrm{e}^{\beta t}t^\nu I_\nu(t)\,\mathrm{d}t\,\mathrm{d}u\\
&\quad+R_1+R_2+R_3,
\end{align*} 
where
\begin{align}|R_3|&=\bigg|\mathrm{e}^{-\beta x}\bigg\{\bigg(\beta+\beta^3+\frac{2\nu+1}{x}\bigg)\frac{K_{\nu}(x)}{x^\nu}+\bigg(1+\beta^2+\frac{\beta(2\nu+1)}{x}\bigg)\frac{K_{\nu+1}(x)}{x^\nu}\bigg\}\nonumber\\
&\quad\times\int_0^xh'(u)\int_0^u \mathrm{e}^{\beta t}t^\nu I_\nu(t)\,\mathrm{d}t\,\mathrm{d}u\bigg|\nonumber\\
&\leq\mathrm{e}^{-\beta x}\bigg\{\bigg(\beta+\beta^3+\frac{2\nu+1}{x}\bigg)\frac{K_{\nu}(x)}{x^\nu}+\bigg(1+\beta^2+\frac{\beta(2\nu+1)}{x}\bigg)\frac{K_{\nu+1}(x)}{x^\nu}\bigg\}\nonumber\\
\label{upbn}&\quad\times\int_0^x\int_0^u \mathrm{e}^{\beta t}t^\nu I_\nu(t)\,\mathrm{d}t\,\mathrm{d}u,
\end{align}
and we used that $\|h'\|=1$ in obtaining the inequality. We have that
\begin{equation}\label{pams2}\int_0^x\int_0^u \mathrm{e}^{\beta t}t^\nu I_\nu(t)\,\mathrm{d}t\,\mathrm{d}u\sim\begin{cases} \displaystyle \frac{x^{2\nu+2}}{2^\nu (2\nu+1)(2\nu+2)\Gamma(\nu+1)}, & \quad x \downarrow 0, \\
\displaystyle \frac{1}{\sqrt{2\pi}(1+\beta)^2}x^{\nu-\frac{1}{2}}\mathrm{e}^{(1+\beta)x}, & \quad x \rightarrow\infty. \end{cases}
\end{equation}
Here the limiting form in the case $x\downarrow0$ readily follows from an application of (\ref{Itend0}), whilst the limiting form for $x\rightarrow\infty$ results from an application of the limiting form (\ref{pamsi}) followed by a standard asymptotic analysis of the integral $\int_0^x u^{\nu-\frac{1}{2}}\mathrm{e}^{(1+\beta) u}\,\mathrm{d}u$ in the limit $x\rightarrow\infty$. Using the limiting form (\ref{pams2}) together with the limiting forms (\ref{Ktend0}) and (\ref{Ktendinfinity}) for the modified Bessel function of the second kind proves that the upper bound (\ref{upbn}) does not blow up in either the limits $x\downarrow0$ or $x\rightarrow\infty$, and can thus be uniformly bounded for all $x\geq0$. Therefore, the term $R_3$ does not explode when $a\rightarrow\infty$.

Finally, we analyse $g^{(3)}(x)$ in a neighbourhood of $x=0$ when $a\rightarrow\infty$. This analysis proceeds almost exactly as this stage of the proof of Proposition 3.6 of \cite{gaunt vgii}, but we repeat the details for completeness. We have shown that, for all $x\geq0$, $R_1$, $R_2$ and $R_3$ are $O(1)$ as $a\rightarrow\infty$.  Therefore using the limiting forms (\ref{Itend0}) and (\ref{Ktend0}) gives
\begin{align*}g^{(3)}(x)&=-\frac{\cos(ax)}{x}+\frac{(2\nu+1)(2\nu+2)}{x^{\nu+2}}\cdot\frac{2^\nu \Gamma(\nu+1)}{x^{\nu+1}}\times \\
&\quad \times\int_0^x \cos(au)\int_0^u\frac{t^{2\nu}}{2^\nu\Gamma(\nu+1)}\,\mathrm{d}t\,\mathrm{d}u+O(1) \\
&=-\frac{\cos(ax)}{x}+\frac{2\nu+2}{x^{2\nu+3}}\int_0^x u^{2\nu+1}\cos(au)\,\mathrm{d}u+O(1), \quad x\downarrow 0.
\end{align*}
As well as letting $x\downarrow0$ and $a\rightarrow\infty$, we let $ax\downarrow0$.  Using the expansion $\cos(t)=1-\frac{1}{2}t^2+O(t^4)$ as $t\downarrow0$, we have, in this regime,
\begin{align*}g^{(3)}(x)&=-\frac{1}{x}\bigg(1-\frac{a^2x^2}{2}\bigg)+\frac{2\nu+2}{x^{2\nu+3}}\int_0^x u^{2\nu+1}\bigg(1-\frac{a^2u^2}{2}\bigg)\,\mathrm{d}u+O(1) \\
&=\frac{a^2x}{2}-\frac{(\nu+1)a^2x}{2\nu+4}+O(1) =\frac{a^2x}{2(\nu+2)}+O(1).
\end{align*}
For $x$ chosen such that $ax\ll1\ll a^2x$, we have that $g^{(3)}(x)$ blows up, and this proves the proposition. \hfill $\Box$

\vspace{3mm}

We will need to following lemma for the proof of Proposition \ref{prop1}. 

\begin{lemma}\label{beslem}(i) Let $0<\nu\leq\frac{1}{2}$.  Then $\mathrm{e}^x x^\nu K_\nu(x)$ is a decreasing function of $x$ on $(0,\infty)$ and satisfies the inequality $\mathrm{e}^x x^\nu K_\nu(x)\leq 2^{\nu-1}\Gamma(\nu)$ for all $x>0$.

%Let $0<\nu\leq\frac{1}{2}$ and $0\leq a<1$.  Then $\mathrm{e}^{ax} x^\nu K_\nu(x)$ is a decreasing function of $x$ and satisfies the inequality $\mathrm{e}^{ax} x^\nu K_\nu(x)\leq 2^{\nu-1}\Gamma(\nu)$ for all $x>0$.
\vspace{2mm}

\noindent{(ii)} Fix $c\geq2$ and let $x_c^*$ be the unique positive solution to $\mathrm{e}^xK_0(x)=-c\log(x)$.  Then, for $0<x<x_c^*$, we have that $\mathrm{e}^x K_0(x)<-c\log(x)$. 

\vspace{2mm}

\noindent{(iii)} Suppose $0<x<0.629$. Then $\mathrm{e}^x K_0(x)<-3\log(x)$.
\end{lemma}

\begin{proof}(i) From (\ref{ddbk}) and (\ref{Kmoni}) we have that $\frac{\mathrm{d}}{\mathrm{d}x}\big(\mathrm{e}^x x^\nu K_\nu(x)\big)=\mathrm{e}^x x^\nu\big(K_\nu(x)- K_{\nu-1}(x)\big)\leq0$, and so $\mathrm{e}^x x^\nu K_\nu(x)$ is a decreasing function of $x$ on $(0,\infty)$.  By (\ref{Ktend0}) we have $\lim_{x\downarrow 0}x^\nu K_\nu(x)=2^{\nu-1}\Gamma(\nu)$, and so we obtain the inequality. 

\vspace{2mm}

\noindent{(ii)} From the differentiation formula (\ref{dkzero}) we have that, for all $x>0$, $\frac{\mathrm{d}}{\mathrm{d}x}\big(-c\log(x)-\mathrm{e}^x K_0(x)\big)=-\frac{c}{x}-\mathrm{e}^x(K_0(x)-K_1(x))<0$, where the inequality follows because $K_1(x)<\frac{2}{x}$ for all $x>0$ (see \cite[Lemma 6.1]{gaunt vgii}).  Thus, $-c\log(x)-\mathrm{e}^x K_0(x)$ is a decreasing function of $x$ on $(0,\infty)$. A simple asymptotic analysis using (\ref{Ktend0}) and (\ref{Ktendinfinity}) shows that $\lim_{x\downarrow0}(-c\log(x)-\mathrm{e}^x K_0(x))>0$ and $\lim_{x\rightarrow\infty}(-c\log(x)-\mathrm{e}^x K_0(x))<0$. The assertion now follows.

\vspace{2mm}

\noindent{(iii)} One can use \emph{Mathematica} to numerically check that $x_3^*=0.62927\ldots>0.629$. 
\end{proof}

\noindent\emph{Proof of Proposition \ref{prop1}.} As usual, we will set $\mu=0$.  In this proof, $Z$ will denote a $\mathrm{VG}(r,\theta,\sigma,0)$ random variable.  For a further simplification, we shall suppose that $\theta\geq0$; the argument for $\theta<0$ is a very similar because $-Z\sim\mathrm{VG}(r,-\theta,\sigma,0)$.

%We shall also set $\sigma=1$.  We can make this simplification because for $Z_\sigma\sim \mathrm{VG}(r,0,\sigma,0)$ we have that $Z_\sigma \stackrel{\mathcal{D}}{=}\sigma Z_1$. Writing $W_\sigma=\sigma W$, by a standard property of the Wasserstein metric,
%\begin{equation}\label{dwe}d_{\mathrm{W}}(W_\sigma,Z_\sigma)=d_{\mathrm{W}}(\sigma W, \sigma Z_1)=\sigma d_{\mathrm{W}}(W,Z_1).
%\end{equation}
%But since $W$ can be any random variable, we can bound $d_{\mathrm{W}}(W,Z_1)$ and then use (\ref{dwe}) to deduce a bound on the desired quantity $d_{\mathrm{W}}(W_\sigma,Z_\sigma)$, which on renaming the arbitrary random variable $W_\sigma$ will complete the proof.  For the rest of the proof, we let $Z\sim \mathrm{VG}(r,0,1,0)$.

\vspace{2mm}

\noindent{(i)} Let $r>1$.  Proposition 1.2 of \cite{ross} asserts that if the random variable $Y$ has Lebesgue density bounded by $C$, then for any random variable $W$,
\begin{equation*}d_{\mathrm{K}}(W,Y)\leq\sqrt{2Cd_{\mathrm{W}}(W,Y)}.
\end{equation*} 
We know that the $\mathrm{VG}(r,\theta,\sigma,0)$ distribution is unimodal \cite{gm20} and that for $r>1$ the density is bounded.  If $1<r\leq2$ the density is bounded above by $C=\frac{1}{2\sigma\sqrt{\pi}}(1+\theta^2/\sigma^2)^{-\frac{r-1}{2}}\Gamma\big(\frac{r-1}{2}\big)/\Gamma\big(\frac{r}{2}\big)$ (see (\ref{pmutend})), and for $r>2$ we can use Proposition \ref{prop2mode} to bound the density.  This gives us the desired bounds.

%Since the $\mathrm{SVG}(r,\sigma,0)$ distribution is unimodal about $0$, it follows from (\ref{pmutend}) that the density is bounded above by $C=\frac{1}{2\sigma\sqrt{\pi}}\Gamma(\frac{r-1}{2})/\Gamma(\frac{r}{2})$, which on substituting into (\ref{subpr1}) yields the desired bound.

\vspace{2mm}

\noindent{(ii)}  We now let $r=1$. We follow the approach used in the proof of Proposition 1.2 of \cite{ross} (see also the proof of Theorem 3.3 of \cite{chen}), but modify part of the argument because the $\mathrm{VG}(1,\theta,\sigma,0)$ density $p(x)$ is unbounded as $x\rightarrow0$. Let $\epsilon>0$ be a constant. Let $h_z(x)=\mathbf{1}(x\leq z)$, and let $h_{z,\epsilon}(x)$ be defined to be one for $x\leq z+2\epsilon$, zero for $x>z$, and linear between.  Then
\begin{align}\mathbb{P}(W\leq z)-\mathbb{P}(Z\leq z)&=\mathbb{E}h_{z}(W)-\mathbb{E}h_{z,\epsilon}(Z)+\mathbb{E}h_{z,\epsilon}(Z)-\mathbb{E}h_z(Z)\nonumber \\
&\leq \mathbb{E}h_{z,\epsilon}(W)-\mathbb{E}h_{z,\epsilon}(Z)+\frac{1}{2}\mathbb{P}(z\leq Z\leq z+2\epsilon)\nonumber \\
&\leq \frac{1}{2\epsilon}d_{\mathrm{W}}(W,Z)+\frac{1}{2}\mathbb{P}(z\leq Z\leq z+2\epsilon)\nonumber \\
\label{puting}&\leq \frac{1}{2\epsilon}d_{\mathrm{W}}(W,Z)+\mathbb{P}(0\leq Z\leq \epsilon),
\end{align}
where the third inequality follows because the $\mathrm{VG}(1,\theta,\sigma,0)$ density is positively skewed about $x=0$ (since $\theta\geq0$), and is a decreasing function of $x$ on $(0,\infty)$ and an increasing function on $(-\infty,0)$.  Suppose $\frac{\epsilon\sqrt{\theta^2+\sigma^2}}{\sigma^2}<0.629$.  Then we can apply part (iii) of Lemma \ref{beslem} to get
\begin{align*}\mathbb{P}(0\leq Z\leq\epsilon)&=\int_0^\epsilon \frac{1}{\pi\sigma}\mathrm{e}^{\theta t/\sigma^2} K_0\bigg(\frac{\sqrt{\theta^2+\sigma^2}}{\sigma^2}t\bigg)\,\mathrm{d}t \leq\frac{1}{\pi\sqrt{\theta^2+\sigma^2}}\int_0^{\frac{\epsilon\sqrt{\theta^2+\sigma^2}}{\sigma^2}}\mathrm{e}^yK_0(y)\,\mathrm{d}y \nonumber\\
&\leq \frac{1}{\pi\sqrt{\theta^2+\sigma^2}}\int_0^{\frac{\epsilon\sqrt{\theta^2+\sigma^2}}{\sigma^2}}-3\log(y)\,\mathrm{d}y =\frac{3\epsilon}{\pi\sigma}\bigg[1+\log\bigg(\frac{\sigma^2}{\epsilon\sqrt{\theta^2+\sigma^2}}\bigg)\bigg].
\end{align*}  
Plugging into (\ref{puting}) gives that, for any $z\in\mathbb{R}$,
\begin{equation*}\mathbb{P}(W\leq z)-\mathbb{P}(Z\leq z)\leq \frac{1}{2\epsilon}d_{\mathrm{W}}(W,Z)+\frac{3\epsilon}{\pi\sigma}\bigg[1+\log\bigg(\frac{\sigma^2}{\epsilon\sqrt{\theta^2+\sigma^2}}\bigg)\bigg].
\end{equation*}
We choose $\epsilon=\sqrt{\pi\sigma d_{\mathrm{W}}(W,Z)/6}$, which, due to the assumption $\frac{\theta^2+\sigma^2}{\sigma^3}d_{\mathrm{W}}(W,Z)<0.755$, guarantees that $\frac{\epsilon\sqrt{\theta^2+\sigma^2}}{\sigma^2}<0.629$.  We therefore obtain the upper bound
\begin{align*}\mathbb{P}(W\leq z)-\mathbb{P}(Z\leq z) &\leq \bigg\{\frac{\sqrt{6}}{2}+\frac{2}{\sqrt{6}}+\frac{1}{\sqrt{6}}\log\bigg(\frac{6\sigma^3}{\pi(\theta^2+\sigma^2)d_{\mathrm{W}}(W,Z)}\bigg)\bigg\}\sqrt{\frac{d_{\mathrm{W}}(W,Z)}{\pi\sigma}} \\
&= \bigg\{5+\log\bigg(\frac{6}{\pi}\bigg)+\log\bigg(\frac{\sigma^3}{(\theta^2+\sigma^2)d_{\mathrm{W}}(W,Z)}\bigg)\bigg\}\sqrt{\frac{d_{\mathrm{W}}(W,Z)}{6\pi\sigma}}.
\end{align*}
A lower bound can be obtained similarly, which is the negative of the upper bound.  This proves inequality (\ref{pronf2}).

%Similarly, we can show that
%\begin{equation*}\mathbb{P}(W\leq z)-\mathbb{P}(Z\leq z)\geq-\bigg\{2+\log\bigg(\frac{2}{\sqrt{\pi}}\bigg)+\frac{1}{2}\log\bigg(\frac{\sigma}{d_{\mathrm{W}}(W,Z)}\bigg)\bigg\}\sqrt{\frac{d_{\mathrm{W}}(W,Z)}{\pi\sigma}}.
%\end{equation*}
%Combining these bounds proves (\ref{pronf2}).

\vspace{2mm}

\noindent{(iii)} Let $0<r<1$. In this regime, the $\mathrm{VG}(r,\theta,\sigma,0)$ density is unbounded as $x\rightarrow0$, positively skewed about $x=0$ (since $\theta\geq0$), and is a decreasing function of $x$ on $(0,\infty)$ and an increasing function on $(-\infty,0)$.  We therefore proceed as we did in part (ii) by bounding $\mathbb{P}(0\leq Z\leq \epsilon)$ and then substituting into (\ref{puting}).  Let $\nu=\frac{r-1}{2}$, meaning that $-\frac{1}{2}<\nu<0$.  Then
\begin{align}\mathbb{P}(0\leq Z\leq \epsilon)&=\frac{1}{\sigma\sqrt{\pi}2^\nu(\theta^2+\sigma^2)^{\frac{\nu}{2}}\Gamma(\nu+\frac{1}{2})}\int_0^\epsilon \mathrm{e}^{\theta t/\sigma^2}t^\nu K_\nu\bigg(\frac{\sqrt{\theta^2+\sigma^2}}{\sigma^2}t\bigg)\,\mathrm{d}t\nonumber \\
&\leq\frac{\sigma^{2\nu+1}}{\sigma\sqrt{\pi}2^\nu(\theta^2+\sigma^2)^{\nu+\frac{1}{2}}\Gamma(\nu+\frac{1}{2})}\int_0^{\frac{\epsilon\sqrt{\theta^2+\sigma^2}}{\sigma^2}} \mathrm{e}^y y^{2\nu}\cdot y^{-\nu} K_{-\nu}(y) \,\mathrm{d}y\nonumber \\
&\leq\frac{\sigma^{2\nu+1}}{\sigma\sqrt{\pi}2^\nu(\theta^2+\sigma^2)^{\nu+\frac{1}{2}}\Gamma(\nu+\frac{1}{2})}\int_0^{\frac{\epsilon\sqrt{\theta^2+\sigma^2}}{\sigma^2}}2^{-\nu-1}\Gamma(-\nu) y^{2\nu} \,\mathrm{d}y\nonumber \\
&=\frac{\Gamma(-\nu)}{\sqrt{\pi}2^{2\nu+1}\Gamma(\nu+\frac{1}{2})}\frac{1}{2\nu+1}\bigg(\frac{\epsilon}{\sigma}\bigg)^{2\nu+1}=:C_{\nu,\sigma}\epsilon^{2\nu+1},\nonumber
\end{align}
where we made a change of variables and applied (\ref{parity}) in the second step, and used Lemma \ref{beslem} in the third.  Therefore, for any $z\in\mathbb{R}$,
\begin{equation*}\mathbb{P}(W\leq z)-\mathbb{P}(Z\leq z)\leq \frac{1}{2\epsilon}d_{\mathrm{W}}(W,Z)+C_{\nu,\sigma}\epsilon^{2\nu+1}.
\end{equation*}
We optimise by taking $\epsilon=\big(\frac{d_{\mathrm{W}}(W,Z)}{2(2\nu+1)C_{\nu,\sigma}}\big)^{\frac{1}{2(\nu+1)}}$, which yields the upper bound
\begin{align*}\mathbb{P}(W\leq z)-\mathbb{P}(Z\leq z)&\leq 2\big(2(2\nu+1)C_{\nu,\sigma}\big)^{\frac{1}{2(\nu+1)}}\big(d_{\mathrm{W}}(W,Z)\big)^{\frac{2\nu+1}{2(\nu+1)}}\\
&=2\bigg(\frac{2\Gamma(-\nu)}{\sqrt{\pi}(2\sigma)^{2\nu+1}\Gamma(\nu+\frac{1}{2})}\bigg)^{\frac{1}{2(\nu+1)}}\big(d_{\mathrm{W}}(W,Z)\big)^{\frac{2\nu+1}{2(\nu+1)}}.
\end{align*}
We can similarly obtain a lower bound, which is the negative of the upper bound. By substituting $\nu=\frac{r-1}{2}$ we obtain (\ref{pronf3}), completing the proof.  \hfill $\Box$ 

\appendix

\section{Elementary properties of modified Bessel functions}\label{appendix}

In this appendix, we present some basic properties of modified Bessel functions that are needed in this paper.  All formulas are given in \cite{olver}, except for the inequalities and the differentiation formulas (\ref{2ndkk})--(\ref{dk3e}).
%, and the differentiation formulas (\ref{2ndkk})--(\ref{dk3}), which are given in \cite{gaunt thesis}, and the differentiation formula (\ref{dk3e}) which we obtain easily from these and other differentiation formulas.

The modified Bessel functions of the first kind $I_\nu(x)$ and second kind $K_\nu(x)$ are defined, for $\nu\in\mathbb{R}$ and $x>0$, by
\[I_{\nu} (x) =  \sum_{k=0}^{\infty} \frac{(\frac{1}{2}x)^{\nu+2k}}{\Gamma(\nu +k+1) k!} \quad \text{and} \quad K_\nu(x)=\int_0^\infty \mathrm{e}^{-x\cosh(t)}\cosh(\nu t)\,\mathrm{d}t.
\]
For $x>0$, the modified Bessel functions $I_\nu(x)$ and $K_\nu(x)$ are strictly positive for $\nu\geq-1$ and all $\nu\in\mathbb{R}$,  respectively.
The modified Bessel function $K_\nu(x)$ satisfies the following identities, which hold for all $\nu\in\mathbb{R}$ and $x\in\mathbb{R}$,
\begin{align} \label{parity}K_{-\nu}(x)&=K_\nu(x), \\
\label{Kiden} K_{\nu+1}(x)&=K_{\nu-1}(x)+\frac{2\nu}{x}K_{\nu}(x).
\end{align}
%We have the following special cases:
%\begin{eqnarray} \label{spheress} I_{1/2}(x)&=&\sqrt{\frac{2}{\pi x}}\sinh(x), \\
% \label{sphere} K_{1/2}(x)&=&\sqrt{\frac{\pi}{2x}} \mathrm{e}^{-x}, \\
% \label{sphere2} K_{3/2}(x)&=& \sqrt{\frac{\pi}{2x}}\bigg(1+\frac{1}{x}\bigg) \mathrm{e}^{-x}.
%\end{eqnarray}
The modified Bessel functions satisfy the following asymptotics:
\begin{eqnarray}
\label{Itend0}I_{\nu} (x) &=& \frac{1}{\Gamma(\nu +1)} \left(\frac{x}{2}\right)^{\nu}\big(1+O(x^2)\big), \quad x \downarrow 0, \\
\label{Ktend0}K_{\nu} (x) &=& \begin{cases} 2^{|\nu| -1} \Gamma (|\nu|) x^{-|\nu|}\big(1+O(x^{p_\nu})\big), &  x \downarrow 0, \: \nu \not= 0, \\
-\log x+O(1), &  x \downarrow 0, \: \nu = 0, \end{cases} \\
 \label{roots} I_{\nu} (x) &=& \frac{\mathrm{e}^x}{\sqrt{2\pi x}}\big(1+O(x^{-1})\big), \quad x \rightarrow \infty,\nonumber  \\
\label{Ktendinfinity} K_{\nu} (x) &=& \sqrt{\frac{\pi}{2x}} \mathrm{e}^{-x}\big(1+O(x^{-1})\big), \quad x \rightarrow \infty.
\end{eqnarray}
Here, $0<p_\nu\leq2$ for all $\nu\not=0$.  In particular, $p_\nu=2$ for $\nu>1$.
%The following differentiation formulas hold:
We also have the following differentiation formulas:
\begin{eqnarray}
\label{dkzero}\frac{\mathrm{d}}{\mathrm{d}x}\big(K_0(x)\big)&=&-K_{1}(x), \\
\label{ddbk}\frac{\mathrm{d}}{\mathrm{d}x}\big(x^\nu K_\nu(x)\big)&=&-x^{\nu} K_{\nu-1}(x), \\
\label{diff11}\frac{\mathrm{d}}{\mathrm{d}x}\bigg(\frac{I_\nu(x)}{x^\nu}\bigg)&=&\frac{I_{\nu+1}(x)}{x^\nu}, \\
\label{diff22}\frac{\mathrm{d}}{\mathrm{d}x}\bigg(\frac{K_\nu(x)}{x^\nu}\bigg)&=&-\frac{K_{\nu+1}(x)}{x^\nu}, \\
\label{2ndkk}\frac{\mathrm{d}^2}{\mathrm{d}x^2}\bigg(\frac{I_{\nu}(x)}{x^{\nu}}\bigg)&=&\frac{I_{\nu}(x)}{x^{\nu}}-\frac{(2\nu+1)I_{\nu+1}(x)}{x^{\nu+1}},  \\
\label{2ndii}\frac{\mathrm{d}^2}{\mathrm{d}x^2}\bigg(\frac{K_{\nu}(x)}{x^{\nu}}\bigg)&=&\frac{K_{\nu}(x)}{x^{\nu}}+\frac{(2\nu+1)K_{\nu+1}(x)}{x^{\nu+1}}, \\
%\label{di3}\frac{\mathrm{d}^3}{\mathrm{d}x^3}\left(\frac{I_{\nu}(x)}{x^{\nu}}\right)&=&\frac{3}{2(\nu+2)}\frac{I_{\nu+1}(x)}{x^{\nu}}+\frac{2\nu+1}{2(\nu+2)}\frac{I_{\nu+3}(x)}{x^{\nu}}, \\
\label{dk3}\frac{\mathrm{d}^3}{\mathrm{d}x^3}\bigg(\frac{K_{\nu}(x)}{x^{\nu}}\bigg)&=&-\frac{(2\nu+1)K_{\nu}(x)}{x^{\nu+1}}-\bigg(1+\frac{(2\nu+1)(2\nu+2)}{x^2}\bigg)\frac{K_{\nu+1}(x)}{x^{\nu}},
\end{eqnarray}
where formulas (\ref{2ndkk})--(\ref{dk3}) are obtained from short calculations that involve differentiating using the formulas (\ref{diff11}) and (\ref{diff22}) followed by an application of the identity (\ref{Kiden}) to simplify the expressions. Using the Leibniz differentiation formula together with formulas (\ref{diff11})--(\ref{dk3}) gives
\begin{align}\label{diffi2e}\frac{\mathrm{d}^2}{\mathrm{d}x^2}\bigg(\frac{\mathrm{e}^{-\beta x}I_{\nu}(x)}{x^{\nu}}\bigg)&=\mathrm{e}^{-\beta x}\bigg\{(1+\beta^2)\frac{I_{\nu}(x)}{x^\nu}-\bigg(2\beta+\frac{2\nu+1}{x}\bigg)\frac{I_{\nu+1}(x)}{x^\nu}\bigg\},\\
\label{diffk2e}\frac{\mathrm{d}^2}{\mathrm{d}x^2}\bigg(\frac{\mathrm{e}^{-\beta x}K_{\nu}(x)}{x^{\nu}}\bigg)&=\mathrm{e}^{-\beta x}\bigg\{(1+\beta^2)\frac{K_{\nu}(x)}{x^\nu}+\bigg(2\beta+\frac{2\nu+1}{x}\bigg)\frac{K_{\nu+1}(x)}{x^\nu}\bigg\},\\
\frac{\mathrm{d}^3}{\mathrm{d}x^3}\bigg(\frac{\mathrm{e}^{-\beta x}K_{\nu}(x)}{x^{\nu}}\bigg)&=-\mathrm{e}^{-\beta x}\bigg\{\bigg(3\beta+\beta^3+\frac{2\nu+1}{x}\bigg)\frac{K_{\nu}(x)}{x^\nu}\nonumber\\
\label{dk3e}&\quad+\bigg(1+3\beta^2+\frac{3\beta(2\nu+1)}{x}+\frac{(2\nu+1)(2\nu+2)}{x^2}\bigg)\frac{K_{\nu+1}(x)}{x^\nu}\bigg\}.
\end{align}

For $x > 0$, the following inequalities hold:
\begin{align}\label{Imon}I_{\nu} (x) &< I_{\nu-1} (x), \quad \nu \geq \tfrac{1}{2}, \\
\label{Kmoni}K_{\nu} (x) &\leq K_{\nu - 1} (x), \quad \nu \leq \tfrac{1}{2},\\
\label{cake}K_{\nu} (x) &\geq K_{\nu - 1} (x), \quad \nu \geq \tfrac{1}{2}.
\end{align}
%\label{cakkk}K_{\nu} (x) &\leq K_{\nu - 1} (x), \quad \nu \leq \tfrac{1}{2},\\
%\label{cake}K_{\nu} (x) &\geq K_{\nu - 1} (x), \quad \nu \geq \tfrac{1}{2}, 
%We have equality in (\ref{cakkk}) and (\ref{cake}) if and only if $\nu=\frac{1}{2}$.  Inequalities (\ref{cakkk}) and (\ref{cake}) for $K_{\nu}(x)$ can be found in \cite{ifantis}.  
We have equality in (\ref{Kmoni}) and (\ref{cake}) if and only if $\nu=\frac{1}{2}$. These two inequalities can be found in \cite{ifantis}. Inequality (\ref{Imon}) is given in \cite{jones} and \cite{nasell}, extending a result of \cite{soni}. Also, we have the following inequality for products of modified Bessel functions, which is given in Corollary 1 of \cite{gaunt ineq2} and is a simple consequence of a monotonicity result of  \cite{penfold}  concerning  the product $K_{\nu}(x)I_\nu(x)$.
% (see also \cite{bar1,bar2} from improvements to the range of validity of the monotonicity result).
%; for other monotonicity results for this product see \cite{bar1,bar2}. 
For $x\geq0$,
\begin{equation}\label{bdsjbc1} K_{\nu}(x)I_\nu(x)\leq\frac{1}{2\nu}, \quad \nu>0.
\end{equation} 
Inequality (D.4) of \cite{gaunt thesis} states that, for $x>0$,
\begin{equation}\label{difiineq}\frac{\mathrm{d}^3}{\mathrm{d}x^3}\bigg(\frac{\mathrm{e}^{-\beta x}I_{\nu}(x)}{x^{\nu}}\bigg)< \frac{8\mathrm{e}^{-\beta x}I_{\nu}(x)}{x^\nu}, \quad\nu>-\tfrac{1}{2},\;-1<\beta<1.
\end{equation}
We also have the following integral inequality, which is a special case of inequality (2.6) of \cite{gaunt ineq1}. For $x\geq0$,
\begin{align}\label{gau11}\int_0^x t^\nu I_\nu(t)\,\mathrm{d}t&\leq\frac{2(\nu+1)}{2\nu+1}x^\nu I_{\nu+1}(x), \quad \nu>-\tfrac{1}{2}.
\end{align}

%For $x>0$, we have the inequality \cite{ifantis}
%\begin{align}\label{Kmoni}K_{\nu} (x) &\leq K_{\nu - 1} (x), \quad \nu \leq \tfrac{1}{2}.
%\label{cake}
%K_{\nu} (x) &\geq K_{\nu - 1} (x), \quad \nu \geq \tfrac{1}{2}. \nonumber 
%\end{align}

\section{Uniform bounds for expressions involving integrals of modified Bessel functions}\label{appb}

In this appendix, we present bounds of \cite{gaunt ineq2, gaunt ineq3,gaunt ineq2020} that we will use to bound the solution of the VG Stein equation. The bounds of \cite{gaunt ineq2, gaunt ineq3,gaunt ineq2020} are stated for the case $\alpha=1$, $-1<\beta<1$; the bounds we state in this appendix follow from a simple change of variables.  
Let $\nu$, $\alpha$ and $\beta$ be such that $\nu>-\frac{1}{2}$ and $|\beta|<\alpha$. Also, let $\gamma=\beta/\alpha$.  We will translate the bounds of \cite{gaunt ineq2, gaunt ineq3,gaunt ineq2020} into the $\mathrm{VG}(r,\theta,\sigma,\mu)$ parametrisation using the change of parameters
\begin{equation*}  \nu=\frac{r-1}{2}, \quad \alpha =\frac{\sqrt{\theta^2 +  \sigma^2}}{\sigma^2}, \quad \beta =\frac{\theta}{\sigma^2}.
\end{equation*}
%We shall also let $\kappa=\sigma^2/\theta^2$.

We first give the following bound, which is not available in the literature, but is easy to derive.  Suppose $\nu>-\frac{1}{2}$ and $0\leq\beta<\alpha$.  Then, for $x\geq0$,
\begin{align}\frac{\mathrm{e}^{-\beta x}K_{\nu+1}(\alpha x)}{x^\nu}\int_0^x \mathrm{e}^{\beta t}t^\nu I_\nu(\alpha t)\,\mathrm{d}t&\leq \frac{K_{\nu+1}(\alpha x)}{x^\nu}\int_0^x t^\nu I_\nu(\alpha t)\,\mathrm{d}t\nonumber\\
\label{gau22}&\leq \frac{2(\nu+1)}{\alpha(2\nu+1)} K_{\nu+1}(\alpha x)I_{\nu+1}(\alpha x)\leq\frac{1}{\alpha(2\nu+1)},
\end{align}
where in the first step we used that $\mathrm{e}^{\beta t}$ is an increasing function of $t$; in the second we used inequality (\ref{gau11}); and in the third we used inequality (\ref{bdsjbc1}).

%Our presentation of inequalities (\ref{vgsolnunibound}) and (\ref{vgsolnunibound1}) differs from that given in \cite{dgv15}.  The bounds are a slight simplification of the bounds given on p$.$ 24 of \cite{dgv15} in a different parametrisation (and translated into our $\mathrm{VG}(r,\theta,\sigma,\mu)$ parametrisation on p$.$ 17 of \cite{dgv15} at the cost of two typos).  One of our simplifications is to note that their $\gamma$ (defined on their p$.$ 24) satisfies $|\gamma|<1$.  The other is note that, for $0<r<2$, 

%=\frac{\sigma^4}{2\sqrt{\theta^2+\sigma^2}(\sqrt{\theta^2+\sigma^2}-|\theta|)}
Suppose now that $\nu>-\frac{1}{2}$ and $|\beta|<\alpha$.  Then, the bounds of \cite{gaunt ineq2, gaunt ineq3,gaunt ineq2020} that we will need are the following. For all $x\geq 0$,
\begin{align}\frac{\mathrm{e}^{-\beta x}K_{\nu}(\alpha x)}{x^\nu}\int_0^x \mathrm{e}^{\beta t}t^{\nu+1}I_\nu(\alpha t)\,\mathrm{d}t&< \frac{1}{2\alpha^2(1-|\gamma|)} \nonumber \\
\label{propb2a12}&=\frac{\sigma^2}{2}\bigg(1+\frac{|\theta|}{\sqrt{\theta^2+\sigma^2}}\bigg)<\sigma^2,\\
\frac{\mathrm{e}^{-\beta x}I_{\nu}(\alpha x)}{x^\nu}\int_x^\infty \mathrm{e}^{\beta t}t^{\nu+1}K_{\nu}(\alpha t)\,\mathrm{d}t&< \frac{1}{\alpha^2}\bigg(1+\frac{2\sqrt{\pi}|\gamma|\Gamma(\nu+\frac{3}{2})}{(1-\gamma^2)^{\nu+\frac{3}{2}}\Gamma(\nu+1)}\bigg)\nonumber\\
&\!\!\!\!\!\!\!\!\!\!\!=\frac{\sigma^4}{\theta^2+\sigma^2}+2\sqrt{\pi}|\theta|\sigma\frac{\Gamma\big(\frac{r}{2}+1\big)}{\Gamma\big(\frac{r+1}{2}\big)}\bigg(1+\frac{\theta^2}{\sigma^2}\bigg)^{\frac{r-1}{2}}\nonumber \\
\label{fff1}&\!\!\!\!\!\!\!\!\!\!\!<\frac{\sigma^4}{\theta^2+\sigma^2}+\sqrt{2\pi}|\theta|\sigma\sqrt{r+1}\bigg(1+\frac{\theta^2}{\sigma^2}\bigg)^{\frac{r-1}{2}}, \\
\label{2020wk1}\frac{\mathrm{e}^{-\beta x}K_{\nu+1}(\alpha x)}{x^\nu}\int_0^x \mathrm{e}^{\beta t}t^{\nu+1}I_\nu(\alpha t)\,\mathrm{d}t&<\frac{1}{2\alpha^2(1-|\gamma|)}<\sigma^2, \\
%\label{2020wk2} \frac{\mathrm{e}^{-\beta x}I_{\nu}(\alpha x)}{x^\nu}\int_x^\infty \mathrm{e}^{\beta t}t^{\nu+1}K_{\nu}(\alpha t)\,\mathrm{d}t& , \\
\label{rnmt2}\frac{\mathrm{e}^{-\beta x}K_{\nu+1}(\alpha x)}{x^\nu}\int_0^x \mathrm{e}^{\beta t}t^\nu I_\nu(\alpha t)\,\mathrm{d}t&\leq \frac{2}{\alpha(2\nu+1)}=\frac{2\sigma^2}{r\sqrt{\theta^2+\sigma^2}},  \\
\label{rnmt2b}\frac{\mathrm{e}^{-\beta x}K_\nu(\alpha x)}{x^\nu}\int_0^x \mathrm{e}^{\beta t}t^\nu I_\nu(\alpha t)\,\mathrm{d}t&\leq \frac{2}{\alpha(2\nu+1)}=\frac{2\sigma^2}{r\sqrt{\theta^2+\sigma^2}},  \\
%\label{rnmt20}\frac{K_\nu(\alpha x)}{x^\nu}\int_0^x t^\nu I_\nu(\alpha t)\,\mathrm{d}t&\leq \frac{1}{\alpha(2\nu+1)}=\frac{\sigma^2}{r\sqrt{\theta^2+\sigma^2}},  \\
\label{rnmt1}\frac{\mathrm{e}^{-\beta x}I_\nu(\alpha x)}{x^\nu}\int_x^\infty \mathrm{e}^{\beta t}t^\nu K_\nu(\alpha t)\,\mathrm{d}t&\leq \frac{M_{\nu,\gamma}}{\alpha}< \frac{\sigma^2A_{r,\theta,\sigma}}{\sqrt{\theta^2+\sigma^2}}, \\
\label{jjj1b0}\frac{\mathrm{e}^{-\beta x}K_{\nu}(\alpha x)}{x^{\nu-1}}\int_0^x \mathrm{e}^{\beta t}t^{\nu}I_\nu(\alpha t)\,\mathrm{d}t&< \frac{2\nu+7}{2\alpha^2(2\nu+1)(1-|\gamma|)}<\sigma^2\bigg(1+\frac{6}{r}\bigg),\\
\label{jjj1b}\frac{\mathrm{e}^{-\beta x}K_{\nu+1}(\alpha x)}{x^{\nu-1}}\int_0^x \mathrm{e}^{\beta t}t^{\nu}I_\nu(\alpha t)\,\mathrm{d}t&< \frac{2\nu+7}{2\alpha^2(2\nu+1)(1-|\gamma|)}<\sigma^2\bigg(1+\frac{6}{r}\bigg),\\
%\label{jjj1}\frac{\mathrm{e}^{-\beta x}K_{\nu}(\alpha x)}{x^{\nu-1}}\int_0^x \mathrm{e}^{\beta t}t^{\nu}I_\nu(\alpha t)\,\mathrm{d}t&< \frac{2\nu+7}{2(2\nu+1)(1-|\gamma|)},  \\
\label{ddd2}\frac{\mathrm{e}^{-\beta x}I_{\nu}(\alpha x)}{x^{\nu-1}}\int_x^\infty \mathrm{e}^{\beta t}t^{\nu}K_{\nu}(\alpha t)\,\mathrm{d}t&< \frac{N_{\nu,\gamma}}{\alpha^2}< \sigma^2 B_{r,\theta,\sigma},   
%\label{ddd2b}\frac{\mathrm{e}^{-\beta x}I_{\nu+1}(\alpha x)}{x^{\nu-1}}\int_x^\infty \mathrm{e}^{\beta t}t^{\nu}K_{\nu}(\alpha t)\,\mathrm{d}t&< N_{\nu,\gamma},
\end{align}
where
\begin{align}\label{mdefn}M_{\nu,\gamma}=\begin{cases} \displaystyle \frac{\sqrt{\pi}\Gamma\big(\nu+\frac{1}{2}\big)}{(1-\gamma^2)^{\nu+\frac{1}{2}}\Gamma(\nu+1)}, & \:  \nu\geq\frac{1}{2}, \\
\displaystyle \frac{6\Gamma(\nu+\frac{1}{2})}{1-|\gamma|}, & \:   |\nu|<\frac{1}{2}, \end{cases}
\end{align}
and
\begin{equation*}N_{\nu,\gamma}= \begin{cases} 
\displaystyle\frac{\sqrt{\pi}\Gamma(\nu+\frac{1}{2})}{(1-\gamma^2)^{\nu+\frac{1}{2}}\Gamma(\nu)}, & \quad \nu\geq\tfrac{1}{2},\\
\displaystyle \frac{1}{1-|\gamma|}, & \quad |\nu|<\tfrac{1}{2}, \end{cases}
\end{equation*}
and in the $\mathrm{VG}(r,\theta,\sigma,\mu)$ parametrisation
\begin{equation*}\label{breqn1appb}
A_{r,\theta,\sigma}=\begin{cases} \displaystyle \frac{2\sqrt{\pi}}{\sqrt{2r-1}}\bigg(1+\frac{\theta^2}{\sigma^2}\bigg)^{\frac{r}{2}}, & \:  r\geq2, \\
\displaystyle 12\Gamma\Big(\frac{r}{2}\Big)\bigg(1+\frac{\theta^2}{\sigma^2}\bigg), & \:   0<r<2, \end{cases}
\end{equation*}
and
\begin{equation*}\label{breqn2appb}B_{r,\theta,\sigma}=\begin{cases} \displaystyle \sqrt{\frac{\pi(r-1)}{2}}\bigg(1+\frac{\theta^2}{\sigma^2}\bigg)^{\frac{r}{2}-1}, & \:  r\geq2, \\
\displaystyle 2, & \:   0<r<2, \end{cases}
\end{equation*}
\begin{comment}
\begin{equation*}\label{breqn2appb0}B_{r,\theta,\sigma}'=\begin{cases} \displaystyle \sqrt{\frac{\pi(r-1)}{2}}\bigg(1+\frac{\theta^2}{\sigma^2}\bigg)^{\frac{r}{2}}, & \:  r\geq2, \\
\displaystyle 2\bigg(1+\frac{\theta^2}{\sigma^2}\bigg), & \:   0<r<2. \end{cases}
\end{equation*}
\end{comment}
which satisfy the inequalities
\begin{equation}\label{ambamb}M_{\nu,\gamma}<A_{r,\theta,\sigma}, \quad N_{\nu,\gamma}<\alpha^2\sigma^2B_{r,\theta,\sigma}.
\end{equation}

Written in the $\mathrm{VG}(r,\theta,\sigma,\mu)$ parametrisation, $M_{\nu,\gamma}$ reads 
\begin{align*}M_{\nu,\gamma}&=\begin{cases} \displaystyle \frac{\sqrt{\pi}\Gamma\big(\frac{r}{2}\big)}{\Gamma\big(\frac{r+1}{2}\big)}\bigg(1+\frac{\theta^2}{\sigma^2}\bigg)^{\frac{r}{2}}, & \:  r\geq2, \\
\displaystyle 6\Gamma\Big(\frac{r}{2}\Big)\frac{\sqrt{\theta^2+\sigma^2}}{\sqrt{\theta^2+\sigma^2}-|\theta|}, & \:   0<r<2. \end{cases} 
\end{align*}
The inequality $M_{\nu,\gamma}<A_{r,\theta,\sigma}$ then follows from an application of the inequality 
\begin{align*}\frac{\sqrt{\theta^2+\sigma^2}}{\sqrt{\theta^2+\sigma^2}-|\theta|}=\frac{\sqrt{\theta^2+\sigma^2}(\sqrt{\theta^2+\sigma^2}+|\theta|)}{\sigma^2}<\frac{2(\theta^2+\sigma^2)}{\sigma^2}=2\bigg(1+\frac{\theta^2}{\sigma^2}\bigg)
\end{align*}
and the upper bound in the two-sided inequality
\begin{equation}\label{oct8}\sqrt{\frac{2}{r}}<\frac{\Gamma(\frac{r}{2})}{\Gamma(\frac{r+1}{2})}<\sqrt{\frac{2}{r-\frac{1}{2}}}, \quad r>1.
\end{equation}
The double inequality (\ref{oct8}) is obtained by comining the inequalities $\frac{\Gamma(x+\frac{1}{2})}{\Gamma(x+1)}>(x+\frac{1}{2})^{-\frac{1}{2}}$ for $x>0$ \cite{gaut}, and $\frac{\Gamma(x+\frac{1}{2})}{\Gamma(x+1)}<(x+\frac{1}{4})^{-\frac{1}{2}}$ for $x>-\frac{1}{4}$ \cite{elezovic}. The inequality $N_{\nu,\gamma}<\alpha^2\sigma^2B_{r,\theta,\sigma}$ can be seen to hold similarly, although this time the lower bound in (\ref{oct8}) is used to bound the ratio of gamma functions, and we use that $\alpha^2\sigma^2=1+\frac{\theta^2}{\sigma^2}$.

In obtaining inequality (\ref{propb2a12}) we calculated
\begin{align*}\frac{1}{\alpha^2(1-|\gamma|)}&=\frac{\sigma^4}{\theta^2+\sigma^2}\frac{\sqrt{\theta^2+\sigma^2}}{\sqrt{\theta^2+\sigma^2}-|\theta|}=\frac{\sigma^4}{\theta^2+\sigma^2}\frac{\sqrt{\theta^2+\sigma^2}(\sqrt{\theta^2+\sigma^2}+|\theta|)}{\sigma^2}\\
&=\sigma^2\bigg(1+\frac{|\theta|}{\sqrt{\theta^2+\sigma^2}}\bigg),
\end{align*}
and in obtaining inequality (\ref{fff1}) we calculated
\begin{align*}\frac{1}{\alpha^2}\frac{|\gamma|\Gamma(\nu+\frac{3}{2})}{(1-\gamma^2)^{\nu+\frac{3}{2}}\Gamma(\nu+1)}&=\frac{\sigma^4}{\theta^2+\sigma^2}\frac{|\theta|}{\sqrt{\theta^2+\sigma^2}}\bigg(1-\frac{\theta^2}{\theta^2+\sigma^2}\bigg)^{-\frac{r}{2}-1}\frac{\Gamma\big(\frac{r}{2}+1\big)}{\Gamma\big(\frac{r+1}{2}\big)}\\
&=\sigma^2\bigg(1+\frac{\theta^2}{\sigma^2}\bigg)^{-1}|\theta|\frac{1}{\sigma}\bigg(1+\frac{\theta^2}{\sigma^2}\bigg)^{-\frac{1}{2}}\bigg(1+\frac{\theta^2}{\sigma^2}\bigg)^{\frac{r}{2}+1}\frac{\Gamma\big(\frac{r}{2}+1\big)}{\Gamma\big(\frac{r+1}{2}\big)}\\
&=|\theta|\sigma\frac{\Gamma\big(\frac{r}{2}+1\big)}{\Gamma\big(\frac{r+1}{2}\big)}\bigg(1+\frac{\theta^2}{\sigma^2}\bigg)^{\frac{r-1}{2}}.
\end{align*}
The final inequality then follows from bounding the ratio of gamma functions using the lower bound in (\ref{oct8}).  All other conversions from the parameters $\nu,\alpha,\beta$ to $r,\theta,\sigma$ are simple and we provide no further details.

In order to obtain our bounds for the solution of the VG Stein equation, we also need some additional bounds that are an easy consequence of some of the inequalities (\ref{propb2a12})--(\ref{ddd2}). To this end, we note two simple inequalities that follow from using the differentiation formulas (\ref{diff11}) and (\ref{diff22}) and the inequalities (\ref{Imon}) and (\ref{cake}), followed by an application of our assumption $|\beta|<\alpha$ to simplify the bound. For $\nu>-\frac{1}{2}$, $|\beta|<\alpha$ and $x>0$,
\begin{align}\label{alph8}\bigg|\frac{\mathrm{d}}{\mathrm{d}x}\bigg(\frac{\mathrm{e}^{-\beta x}I_{\nu}(\alpha x)}{x^{\nu}}\bigg)\bigg|=\frac{\mathrm{e}^{-\beta x}}{x^\nu}\big|\alpha I_{\nu+1}(x)-\beta I_{\nu}(x)\big|&<2\alpha\frac{\mathrm{e}^{-\beta x}I_{\nu}(\alpha x)}{x^{\nu}}, \\
\label{alph9}\bigg|\frac{\mathrm{d}}{\mathrm{d}x}\bigg(\frac{\mathrm{e}^{-\beta x}K_{\nu}(\alpha x)}{x^{\nu}}\bigg)\bigg|=\frac{\mathrm{e}^{-\beta x}}{x^\nu}\big|\alpha K_{\nu+1}(x)+\beta K_{\nu}(x)\big|&<2\alpha\frac{\mathrm{e}^{-\beta x}K_{\nu+1}(\alpha x)}{x^{\nu}}.
\end{align}
If we restrict to $-\alpha<\beta\leq0$, then we can improve (\ref{alph9}) to
\begin{equation}
\label{alph90}\bigg|\frac{\mathrm{d}}{\mathrm{d}x}\bigg(\frac{\mathrm{e}^{-\beta x}K_{\nu}(\alpha x)}{x^{\nu}}\bigg)\bigg|\leq\alpha\frac{\mathrm{e}^{-\beta x}K_{\nu+1}(\alpha x)}{x^{\nu}}.
\end{equation}
Combining inequalities (\ref{alph8}) and (\ref{alph9}) with certain bounds from the list (\ref{propb2a12})--(\ref{ddd2}) then yields the following uniform bounds (\ref{jjj1z0})--(\ref{ddd2z}).  Suppose $\nu>-\frac{1}{2}$ and $|\beta|<\alpha$. Then, for any $x\geq0$,
\begin{align}\bigg|\frac{\mathrm{d}}{\mathrm{d}x}\bigg(\frac{\mathrm{e}^{-\beta x}K_{\nu}(\alpha x)}{x^{\nu}}\bigg)\bigg|\int_0^x \mathrm{e}^{\beta t}t^{\nu+1}I_\nu(\alpha t)\,\mathrm{d}t&<\frac{1}{\alpha(1-|\gamma|)}\nonumber \\
\label{jjj1z0}&=\sqrt{\theta^2+\sigma^2}+|\theta| <2\sqrt{\theta^2+\sigma^2}, \\
\bigg|\frac{\mathrm{d}}{\mathrm{d}x}\bigg(\frac{\mathrm{e}^{-\beta x}I_{\nu}(\alpha x)}{x^{\nu}}\bigg)\bigg|\int_x^\infty \mathrm{e}^{\beta t}t^{\nu+1}K_{\nu}(\alpha t)\,\mathrm{d}t&<\frac{2}{\alpha}\bigg(1+\frac{2\sqrt{\pi}|\gamma|\Gamma(\nu+\frac{3}{2})}{(1-\gamma^2)^{\nu+\frac{3}{2}}\Gamma(\nu+1)}\bigg)\nonumber \\
&\!\!\!\!\!\!\!\!\!\!\!\!\!\!\!\!\!\!\!\!\!\!\!\!\!\!\!\!\!\!\!\!\!\!\!\!<2\frac{\sqrt{\theta^2+\sigma^2}}{\sigma^2}\Bigg\{\frac{\sigma^4}{\theta^2+\sigma^2}+\sqrt{2\pi}|\theta|\sigma\sqrt{r+1}\bigg(1+\frac{\theta^2}{\sigma^2}\bigg)^{\frac{r-1}{2}}\Bigg\}\nonumber \\
\label{ddd2z0}&\!\!\!\!\!\!\!\!\!\!\!\!\!\!\!\!\!\!\!\!\!\!\!\!\!\!\!\!\!\!\!\!\!\!\!\!=\frac{2\sigma^2}{\sqrt{\theta^2+\sigma^2}}+2\sqrt{2\pi}|\theta|\sqrt{r+1}\bigg(1+\frac{\theta^2}{\sigma^2}\bigg)^{\frac{r}{2}}, \\
%\label{jjj1z00}\bigg|\frac{\mathrm{d}}{\mathrm{d}x}\bigg(\frac{\mathrm{e}^{-\beta x}K_{\nu}(\alpha x)}{x^{\nu}}\bigg)\bigg|\int_0^x \mathrm{e}^{\beta t}t^{\nu}I_\nu(\alpha t)\,\mathrm{d}t&\leq\frac{2}{2\nu+1}=\frac{2}{r} , \\
\label{ddd2z00}\bigg|\frac{\mathrm{d}}{\mathrm{d}x}\bigg(\frac{\mathrm{e}^{-\beta x}I_{\nu}(\alpha x)}{x^{\nu}}\bigg)\bigg|\int_x^\infty \mathrm{e}^{\beta t}t^{\nu}K_{\nu}(\alpha t)\,\mathrm{d}t&<2M_{\gamma,\nu}<2A_{r,\theta,\sigma}, 
\end{align}
\begin{align}
x\bigg|\frac{\mathrm{d}}{\mathrm{d}x}\bigg(\frac{\mathrm{e}^{-\beta x}K_{\nu}(\alpha x)}{x^{\nu}}\bigg)\bigg|\int_0^x \mathrm{e}^{\beta t}t^{\nu}I_\nu(\alpha t)\,\mathrm{d}t&< \frac{2\nu+7}{\alpha(2\nu+1)(1-|\gamma|)} \nonumber \\
\label{jjj1z}&<2\bigg(1+\frac{6}{r}\bigg)\sqrt{\theta^2+\sigma^2}, \\
\label{ddd2z}x\bigg|\frac{\mathrm{d}}{\mathrm{d}x}\bigg(\frac{\mathrm{e}^{-\beta x}I_{\nu}(\alpha x)}{x^{\nu}}\bigg)\bigg|\int_x^\infty \mathrm{e}^{\beta t}t^{\nu}K_{\nu}(\alpha t)\,\mathrm{d}t&< \frac{2N_{\nu,\gamma}}{\alpha}< 2\sqrt{\theta^2+\sigma^2}B_{r,\theta,\sigma}.
\end{align}
We also have the bound
\begin{equation}\label{jjj1z00}\bigg|\frac{\mathrm{d}}{\mathrm{d}x}\bigg(\frac{\mathrm{e}^{-\beta x}K_{\nu}(\alpha x)}{x^{\nu}}\bigg)\bigg|\int_0^x \mathrm{e}^{\beta t}t^{\nu}I_\nu(\alpha t)\,\mathrm{d}t\leq\frac{2}{2\nu+1}=\frac{2}{r}.
\end{equation}
For $-\alpha<\beta\leq0$, this bound follows from combining inequalities (\ref{rnmt2}) and (\ref{alph90}). For $0\leq\beta<\alpha$, we combine inequalities (\ref{gau22}) and (\ref{alph9}).

\subsection*{Acknowledgements}
The author is supported by a Dame Kathleen Ollerenshaw Research Fellowship. I would like to thank Ehsan Azmoodeh for a helpful discussion. I would like to thank the reviewers for their helpful comments and suggestions.
%and for sharing a draft of a forthcoming work that utilises theory developed in this paper.

\end{document}